\documentclass[a4,10pt]{article}

\usepackage{amssymb}
\usepackage{amsthm}
\usepackage{bm}
\usepackage{todonotes}
\usepackage{booktabs}
\usepackage{csquotes}
\usepackage{mathtools}
\usepackage{subcaption}
\usepackage{enumitem}
\setenumerate{label=\alph*)}
\usepackage[percent]{overpic}
\usepackage[pdftex,colorlinks=true]{hyperref}
\usepackage{authblk}
\usepackage{chngcntr}
\usepackage{apptools}
\usepackage[margin=1.2in]{geometry}
\AtAppendix{\counterwithin{lemma}{section}}

\theoremstyle{definition}
\newtheorem{thm}{Theorem}
\newtheorem{defn}[thm]{Definition}
\newtheorem{lem}[thm]{Lemma}

\newtheorem{rem}[thm]{Remark}
\newtheorem{cor}[thm]{Corollary}

\newcommand{\ii}{\mathrm{i}}

\newcommand{\bc}{\mathbf c}

\renewcommand{\bf}{\mathbf f}

\newcommand{\bz}{\mathbf z}

\newcommand{\bx}{\mathbf x}
\newcommand{\bn}{\mathbf n}
\newcommand{\bg}{\mathbf g}
\newcommand{\bA}{\mathbf A}
\newcommand{\bC}{\mathbf C}

\newcommand{\bN}{\mathbf N}
\newcommand{\bB}{\mathbf B}
\newcommand{\bS}{\mathbf S}
\newcommand{\bT}{\mathbf T}

\newcommand{\bD}{\mathbf D}

\newcommand{\bI}{\mathbf I}

\newcommand{\bF}{\mathbf F}
\newcommand{\bff}{\mathbf f}
\newcommand{\bG}{\mathbf G}

\newcommand{\bzero}{\mathbf 0}
\newcommand{\bv}{\mathbf v}
\newcommand{\by}{\mathbf y}

\newcommand{\bw}{\mathbf w}

\newcommand{\R}{\mathbb R}
\newcommand{\C}{\mathbb C}

\newcommand{\tm}{\subseteq}
\newcommand{\abs}[1]{\lvert #1\rvert}
\newcommand{\norm}[1]{\lVert #1\rVert}

\newcommand{\be}{\mathbf e}

\newcommand{\qta}{\quad\text{ and }\quad}

\renewcommand{\R}{\mathbb{R}}
\newcommand{\N}{\mathbb{N}}

\newcommand{\bby}{\bar{\by}}
\renewcommand{\Vec}[1]{\renewcommand*{\arraystretch}{1.2}\begin{pmatrix*}[r]#1\end{pmatrix*}}
\renewcommand{\vec}[1]{\begin{psmallmatrix}#1\end{psmallmatrix}}

\newcommand{\from}{\colon}

\DeclareMathOperator{\diag}{diag}
\DeclareMathOperator{\tr}{trace}
\DeclareMathOperator{\im}{Im}
\DeclareMathOperator{\re}{Re}
\DeclareMathOperator{\Span}{span}
\DeclareMathOperator{\rank}{rank}
\DeclareMathOperator{\diff}{d}

\newcommand{\suggest}[1]{{\color{black}{#1}}}

\makeatletter
\newcommand{\vast}{\bBigg@{3}}
\newcommand{\Vast}{\bBigg@{4}}
\makeatother

\newcommand{\ns}{\mkern-5mu}

\makeatletter
\newif\ifcenter@asb@\center@asb@false
\def\center@arstrutbox{%
    \setbox\@arstrutbox\hbox{$\vcenter{\box\@arstrutbox}$}%
    }
\newcommand*{\CenteredArraystretch}[1]{%
    \ifcenter@asb@\else
      \pretocmd{\@mkpream}{\center@arstrutbox}{}{}%
      \center@asb@true
    \fi
    \renewcommand{\arraystretch}{#1}%
    }
\makeatother

\usepackage{tikz}
\usetikzlibrary{angles,quotes}
\usepackage{tkz-euclide}

\definecolor{colorA}{rgb}{0,0.447,0.741}
\definecolor{colorB}{rgb}{0.85,0.325,0.098}
\definecolor{colorE}{rgb}{0.929,0.694,0.125}
\definecolor{colorF}{rgb}{0.494,0.184,0.556}
\definecolor{colorD}{rgb}{0.466,0.674,0.188}
\definecolor{colorC}{rgb}{0.301,0.745,0.933}
\definecolor{colorG}{rgb}{0.635,0.078,0.184}

\makeatletter
\newcommand{\subalign}[1]{%
	\vcenter{%
		\Let@ \restore@math@cr \default@tag
		\baselineskip\fontdimen10 \scriptfont\tw@
		\advance\baselineskip\fontdimen12 \scriptfont\tw@
		\lineskip\thr@@\fontdimen8 \scriptfont\thr@@
		\lineskiplimit\lineskip
		\ialign{\hfil$\m@th\scriptstyle##$&$\m@th\scriptstyle{}##$\hfil\crcr
			#1\crcr
		}%
	}%
}
\makeatother

\makeatletter
\usepackage{environ} 
\newsavebox{\measure@tikzpicture}
\NewEnviron{scaletikzpicturetowidth}[1]{%
	\def\tikz@width{#1}%
	\def\tikzscale{1}\begin{lrbox}{\measure@tikzpicture}%
		\BODY
	\end{lrbox}%
	\pgfmathparse{#1/\wd\measure@tikzpicture}%
	\edef\tikzscale{\pgfmathresult}%
	\BODY
}
\makeatother

\begin{document}



\title{\suggest{On the Dynamics} of First and Second Order \\ GeCo and gBBKS Schemes}

\author[1]{Thomas Izgin} 
\author[2,3]{Angela Martiradonna}
\author[1]{Stefan Kopecz}
\author[1]{Andreas Meister}
\affil[1]{Department of Mathematics and Natural Sciences, University of Kassel, Germany}
\affil[2]{University of Bari Aldo Moro, Department of Mathematics, Italy}
\affil[3]{Istituto per le Applicazioni del Calcolo M. Picone, CNR, italy}
\affil[1]{izgin@mathematik.uni-kassel.de\ \&\ kopecz@mathematik.uni-kassel.de\ \&\ meister@mathematik.uni-kassel.de}
\affil[2,3]{a.martiradonna@ba.iac.cnr.it}
\setcounter{Maxaffil}{0}
\renewcommand\Affilfont{\itshape\small}
\maketitle
\begin{abstract}
In this paper we investigate the stability properties of the so-called gBBKS and GeCo methods, which belong to the class of nonstandard schemes and preserve the positivity as well as all linear invariants of the underlying system of ordinary differential equations for any step size.
\suggest{A stability investigation for these methods, which are outside the class of general linear methods, is challenging since the iterates are always generated by a nonlinear map even for linear problems. Recently, a stability theorem was derived presenting criteria for understanding such schemes.} 

\suggest{For the analysis, the} schemes are applied to general linear equations and proven to be generated by $\mathcal C^1$-maps with locally Lipschitz continuous first derivatives. As a result, \suggest{the above mentioned} stability theorem can be applied to investigate the Lyapunov stability of non-hyperbolic fixed points of the numerical method by analyzing the spectrum of the corresponding Jacobian of the generating map. In addition, if a fixed point is proven to be stable, the theorem guarantees the local convergence of the iterates towards it.

In the case of first and second order gBBKS schemes the stability domain coincides with that of the underlying Runge--Kutta method. Furthermore, while the first order GeCo scheme converts steady states to stable fixed points for all step sizes and all linear test problems of finite size, the second order GeCo scheme has a bounded stability region for the considered test problems. Finally, all theoretical predictions from the stability analysis are validated numerically.
\end{abstract}


\section{Introduction}
\label{sec:Intro}

Many realistic phenomena in biology, chemistry, epidemiology and ecology are  modeled  by systems of differential equations that are constrained by restrictions linked to the nature of the problem \cite{lacitignola2021using, colonna2016plasma,  kooijman2000dynamic}.  
Such systems are often featured by positive state variables and by the conservation of some linear invariants, such
as the total density. 

A general biochemical system \cite{formaggia2011positivity,BBKS2007} is defined as a system \suggest{of ordinary differential equations}
\begin{equation}\label{eq:bio}
	\by'=\mathbf{f}(\by),  \qquad \by(0)=\suggest{\by^0},
\end{equation}
where $\mathbf{y}=(y_1,\dots,y_N)^\mathsf{T}$ denotes the vector of the state variables and the vector field is given by $\bf(\by)=\bS\mathbf r(\by)$. Here,  $\bS\in\mathbb R^{N\times M}$ is the stoichiometric matrix with entries $s_{ij}$ for $i=1,\dotsc, N$ and $j=1,\dotsc,M$, and $\mathbf{r}(\mathbf{y})=(r_1(\mathbf{y}),\dots,r_M(\mathbf{y}))^\mathsf{T}$ is the vector of the reaction functions. The following assumptions, stated in \cite{formaggia2011positivity}, assure the well-posedness of the system and the positivity of the solutions:
\begin{enumerate}
	\item $r_j\in \mathcal{C}^0 \left(\overline{\mathbb{R}}_+^N           \right)$ is locally Lipschitz in $\mathbb{R}^N$ for $j=1,\dots,M$;
	\item $\mathbf{r}(\mathbf{y})>\mathbf 0$ if $\mathbf{y}>\mathbf 0$ and                $\mathbf{r}(\mathbf{y})=\mathbf 0$ if $\mathbf{y}=\mathbf 0$;
	\item If $s_{ij}<0$, there exists a 
	$q_j\in \mathcal{C}^0 \left(\overline{\mathbb{R}}_+^N \right)$
	such that $r_j(\mathbf{y})=q_j(\mathbf{y})y_i$.
\end{enumerate}

All linear functions $\xi_{\mathbf n}(\mathbf{y})=\mathbf{n}^\mathsf{T}\mathbf{y}$, with  $\mathbf{n}\in \ker(\bS^\mathsf{T})=\left\{ \mathbf n \in \mathbb{R}^N: \mathbf n^\mathsf{T} \bS=\bzero
\right\}$
are first integrals for (\ref{eq:bio}), as $\tfrac{\diff\xi_{\mathbf n}(\by)}{\diff t} = \mathbf n ^\mathsf{T} \tfrac{\diff\mathbf{y}}{\diff t} =  \mathbf n^\mathsf{T}\bS \mathbf{r} (\mathbf{y}) =0$. This means that, if $\rank(\bS)<N$,  the system (\ref{eq:bio}) possesses $k=N-\rank(\bS)$ independent linear invariants $\xi_{\suggest{\bn}_j}=\mathbf{n}_j^\mathsf{T} \by$, $j=1,\dots,k$, where $\{\mathbf{n}_1,\dots,\mathbf{n}_k\}$ is a basis of $\ker(\bS^\mathsf{T})$.  

In the general principle of geometric numerical integration \cite{hairer2006geometric}, numerical schemes for systems of differential equations should provide approximate solutions that are featured by the same geometric properties of the exact flow \cite{diele2019geometric}. In particular, geometric numerical integrators for (\ref{eq:bio}) are required \suggest{to preserve all linear invariants of the system and} to be unconditionally positive, \suggest{i.\,e.\ }the numerical solution is positive \suggest{for any positive} step size.
Experiments suggest that the use of positive and linear invariants preserving numerical integrators is mandatory in order to adequately detect the behavior of the dynamical system and disregarding them can also lead to incorrect equilibrium states \cite{lacitignola2021using,shampine1986conservation}.
Moreover, positive schemes are of practical
interest  whenever the loss of positivity in the approximate solutions induces instability \cite{sandu2001positive,lacitignola2021using}.

While high order linear integrators, namely Runge--Kutta and linear multistep methods, preserve exactly all the linear invariants of the system, unconditional positivity is much harder to obtain. Among the class of linear integrators, unconditional positivity is restricted to first order \cite{MR1963511}. The implicit Euler method indeed grants the positivity, although methods for solving nonlinear systems coming from implicit schemes do not guarantee positive approximations. 
Higher order linear methods can only guarantee positivity by restricting the time step size, leading to
a significant increase in computational time \cite{MR1963511,bertolazzi1996positive}.

Positive and linear invariants preserving  schemes based on projection techniques were proposed  in \cite{sandu2001positive, nusslein2021positivity}, where at each time step, the negative approximations or the weights of the Runge--Kutta method are changed to guarantee positivity while maintaining the order of the method.  More recently, the issue of positivity preservation was addressed in \cite{blanes2021positivity}, where splitting and exponential methods were combined to construct positive and conservative integrators up to third order for \suggest{solving} nonlinear mass conservative system\suggest{s} of the type $\by'(t)=\bA(t,\by(t))\by(t)$, where $\bA(t,\by(t))$ is a $N\times N$ matrix-valued function. 

\suggest{In the context of production-destruction systems (PDS), where  the only linear invariant to be preserved is the total density, modified Patankar--Runge--Kutta (MPRK) methods, originally introduced in \cite{BDM2003}, have been of considerable interest in recent years. Second and third order MPRK schemes have been developed in  \cite{KM18,KM18Order3,MR3987250} and the idea was then carried out in the context of SSP Runge--Kutta methods in \cite{MR3969000,MR3934688}, where they have been applied to solve reactive Euler equations. The so-called Patankar-trick was also used in \cite{BDM2003} to develop MPDeC schemes, which are modified Patankar schemes of arbitrary order based on deferred correction schemes. All these schemes are mass conservative and unconditionally positive. Moreover  their efficiency and robustness was proven numerically while integrating stiff PDS.}

Among the positive and linear invariants preserving integrators for biochemical systems, first and second order  gBBKS \cite{BBKS2007,BRBM2008,MR4109346} and GeCo schemes \cite{martiradonna2020geco} have been introduced in recent literature. 
These schemes fall in the class of  nonstandard integrators \cite{mickens1994nonstandard}, as they result as  nonstandard versions of explicit first and second order Runge--Kutta schemes, where the advancement in time is modulated by a nonlinear functional dependency on the temporal step size and on the  approximation itself. 
The step size modification \suggest{thereby} guarantees the numerical solution to be unconditionally positive while keeping the accuracy of the underlying method. 
While GeCo schemes are explicit integrators, the gBBKS step size modification function leads to an implicit scheme.
Nevertheless, nonlinear implicit equations that arise from gBBKS schemes may be reduced to a  scalar nonlinear equation in one single unknown \suggest{\cite{MR4109346}}.

Previous studies on stability of nonstandard schemes were done in \cite{dimitrov2007stability, dimitrov2008nonstandard, wood2017universal}, where some conditions on the step size modification function were deduced in order to guarantee the elementary stability of first and second order nonstandard Runge--Kutta schemes applied to  autonomous systems with a finite number of hyperbolic equilibria\suggest{. Moreover,}  positive and elementary stable nonstandard methods were constructed for specific systems used to model physical processes and chemical reactions as well as biological interactions and epidemic models. 
The stability of nonstandard schemes applied to systems with infinitely many non-hyperbolic steady state deserves more attention and needs to be better investigated, since already in any linear test problem the presence of linear invariants implies the existence of infinitely many non-hyperbolic fixed points of the method. \suggest{Getting a deeper insight in the stability of steady states of the differential equations or the corresponding fixed points of the method is an important step towards understanding the dynamics of nonlinear equations and methods, respectively. The stability analysis of hyperbolic steady states and fixed points can be reduced to an eigenvalue problem \cite{SH98}, while for the non-hyperbolic setting more assumptions on the underlying differential equations are needed to overcome a case by case study.}

In \cite{IKM2122}, \suggest{such} a  stability analysis for positive and mass conservative time integration schemes  applied to two dimensional linear PDS was developed and  used for exploring the stability properties of second order MPRK schemes. The mass conservation law in PDS leads to the presence of a non-hyperbolic steady state in the system and, consequently, in the  presence of a non-hyperbolic fixed point in the numerical schemes. Therein, the center manifold theory for maps \cite{carr1982,mccracken1976hopf} is  used to derive sufficient conditions for the stability as well as local convergence of $\mathcal C^2$ iteration schemes for PDS.

Then,\suggest{ the stability} Theorem 2.9 in \cite{IKM2122} has been extended in \cite{izgin2022stability} 
to higher dimensional linear systems with some linear invariants, which are characterized by the presence of non-hyperbolic steady states. Based on the center manifold theory, a criterion to assess the stability of a positive and linear invariants preserving scheme $\by^{n+1}=\bg(\by^n)$ is given, provided that $\bg\in \mathcal C^1$ as well as that the first derivatives of $\bg$ are Lipschitz continuous. This criterion was then used to investigate second order MPRK schemes when applied to arbitrary finite sized linear system in \cite{izgin2022stability} as well as for the analysis of higher order SSPMPRK \cite{HIKMS22}, MPRK and MPDeC methods \cite{IOE22StabMP}. The obtained stability functions were used to derive a necessary condition for the time step size to avoid oscillatory behavior of the numerical approximation, see \cite{TORLO,ITOE22} for more insights into this topic. Moreover, the stability theory from \cite{IKM2122, izgin2022stability} was also used to investigate the stability of MPRK22($\alpha$) when applied to a nonlinear test problem, see  \cite{IKMnonlin22}.

In this paper, we use the theory developed in \cite{izgin2022stability} in order to explore the stability properties of gBBKS and GeCo schemes.
\suggest{We explore the stability of the first order GeCo (GeCo1)  scheme in the general $N$-dimensional linear setting
\begin{equation*}
	\by'(t)=\bA\by(t),\qquad \by(0)=\by^0,
\end{equation*}
where $\bA\in\mathbb{R}^{N\times N}$ has $\rank(\bA)<N$ and is a Metzler matrix, i.\,e.\ $a_{ij}\geq 0$ for $i\neq j$. As a result the system  possesses $k=N-\rank(\bA)\geq 1$ linear invariants and its solution is positive whenever $\by^0>0$. Moreover, to ensure stable steady states $\by^*\in \ker(\bA)$,  the matrix $\bA$ must have a spectrum
$\sigma(\bA) \tm \C^- = \{z \in \mathbb \C : \re(z) \leq 0\}$ 
and the eigenvalues of $\bA$ with vanishing real part have to be associated with a Jordan block
size of 1, see \cite[Theorem 3.23]{MR1912409}.}

\suggest{Due to the complex structure of the gBBKS schemes and  the second order GeCo scheme, we} search for the\suggest{ir} stability regions when applied to the two dimensional linear test problem
\label{sec:TestProb.StabTheorem}
\begin{align}
	\by'(t)=\bA\by(t),\quad 
	\by(t)=\left[\begin{array}{c} y_1(t)\\y_2(t)\end{array}\right],
	\quad \bA=\begin{pmatrix*}[r]
		-ac &bc\\ a &-b
	\end{pmatrix*},\quad a,b,c> 0\label{PDS_test}
\end{align}
with initial condition 
\begin{equation}
	\by(0)=\by^0=\vec{y_1^0\\y_2^0}>\bzero.\label{eq:IC}
\end{equation}
The eigenvalues of $\bA$ are given by $0$ and $\lambda=-(ac+b)<0$, and the steady states of \eqref{PDS_test} are given by $\ker(\bA)=\Span(\by^*)$ with $\by^*=(b,a)^\mathsf{T}$.
Such a system provides positive solutions whenever $y_1^0,y_2^0>0$ and possesses $ y_1(t)+cy_2(t)=y_1^0+cy_2^0$ for any $t\geq 0$ as \suggest{the} only linear invariant. \suggest{Moreover, it is known from \cite{IKM21} that all steady states of \eqref{PDS_test} are Lyapunov stable but none of them is asymptotically stable. Therefore, we expect to observe the same stability properties for the corresponding fixed points using a reasonable method.} In order to apply Theorem 2.9 from \cite{izgin2022stability}, we prove that the iterates of the numerical schemes are generated by $\mathcal C^1$-maps with Lipschitz continuous first derivatives when applied to the test problem \eqref{PDS_test}. \suggest{As we will see, already for these $2\times 2$ systems, the first and second order gBBKS as well as the second order GeCo scheme possess only bounded stability domains, disqualifying them for solving stiff problems. }

The paper is organized as follows. In the next section, we shortly present definitions and results for the stability analysis of fixed points. \suggest{In Section \ref{sec:GeCo}, we apply the stability theory from \cite{izgin2022stability} to first and second order GeCo schemes from \cite{martiradonna2020geco}.} Next, in Section \ref{sec:BBKS} we investigate the stability properties of first and second order generalized BBKS schemes from \cite{BBKS2007,BRBM2008, MR4109346}. We \suggest{confirm} the theoretical results \suggest{by} numerical experiments in Section \ref{sec:NumTests}.
Finally, in Section \ref{sec:Summary}, we formulate our summary and conclusions.

\section{Stability Theorem}


We recall in this section the definition and characterization of Lyapunov stability of  fixed points of maps along with  the main theorem of \cite{izgin2022stability}, which provides criteria to assess the stability of non-hyperbolic
fixed points of nonlinear iterations conserving some linear invariant.
In the following, we use $\Vert \cdot \Vert$ to represent an arbitrary norm in $\mathbb R^N$ and $\bD\bg$ denotes the Jacobian of a map $\bg\colon \R^N\to \R^N$.

\begin{defn}\label{Def_Lyapunov_Diskr}
	Let $\by^*$ be a fixed point of an iteration scheme $\by^{n+1}=\bg(\by^n)$, that is $\by^*=\bg(\by^*)$. 
	\begin{enumerate}
		\item\label{def:stab} The fixed point $\by^*$ is called \emph{Lyapunov stable} if, for any $\epsilon>0$, there exists a $\delta=\delta(\epsilon)>0$ such that $\norm{\by^0-\by^*}<\delta$ implies $\norm{\by^n- \by^*}<\epsilon$ for all $n\geq 0$.
		\item If in addition to a), there exists a constant $c>0$ such that $\Vert \by^0-\by^*\Vert<c$ implies $\Vert \by^n-\by^*\Vert \to 0$ for $n\to \infty$, the fixed point $\by^*$ is called \emph{asymptotically stable.}
		\item A fixed point that is not Lyapunov stable is said to be \emph{unstable}.
	\end{enumerate}
\end{defn}
In the following, we will also briefly speak of stability instead of Lyapunov stability. \suggest{According to this definition, the iterates of a method starting close enough to a stable fixed point will stay within an $\epsilon$-neighborhood of the fixed point. Of course, this property is only desirable, if the continuous time problem possesses steady state solutions with identical properties. As already mentioned in the introduction, the steady states of the test problem \eqref{PDS_test} are stable and not asymptotically stable. As we will see, the general linear test problem will share the same features concerning the stability of its steady states.}

To decide if a fixed point $\mathbf y^*$ of a map $\mathbf g$ is stable, we can compute the absolute values of the eigenvalues of the Jacobian $\mathbf D\mathbf g(\mathbf y^*)$. The following theorem characterizes the stability of a fixed point, if all eigenvalues $\lambda$ of $\mathbf D\mathbf g(\mathbf y^*)$ satisfy $\abs\lambda \ne 1$. 
\begin{thm}[{\cite[Theorem 1.3.7]{SH98}}]\label{Thm:_Asym_und_Instabil}
	Let  $\bg\in \mathcal C^1$ and $\by^{n+1}=\bg(\by^n)$ be an iteration scheme with fixed point $\by^*$. Then
	\begin{enumerate}
		\item $\by^*$ is asymptotically stable if $\abs\lambda <1$ for all eigenvalues $\lambda$ of $\bD\bg(\by^*)$. 
		\item $\by^*$ is unstable if $\abs {\lambda}>1$ for at least one eigenvalue $\lambda$ of $\bD\bg(\by^*)$.
	\end{enumerate}
\end{thm}
The above theorem makes no statement for cases in which $\mathbf D\mathbf g(\mathbf y^*)$ has a spectral radius of $1$\suggest{, which is related to the case of stable fixed points that are not asymptotically stable}. In such a case the following Theorem~\ref{Thm_MPRK_stabil} can be helpful. To formulate this theorem, we introduce the following notations.
Let $\bA\in \R^{N\times N}$ such that $\bn_1,\dotsc,\bn_k$ with $k\geq 1$ form a basis of $\ker(\bA^\mathsf{T})$, we define the matrix
\begin{equation*}\label{eq:N}
	\bN=\begin{pmatrix}
		\bn_1^\mathsf{T}\\
		\vdots\\
		\bn_k^\mathsf{T}
	\end{pmatrix}\in \R^{k\times N} 
\end{equation*}
as well as the set
\begin{equation}
	H=\{\by\in \R^N\mid \bN\by=\bN\by^*\}\suggest{,}\label{eq:H}
\end{equation}
\suggest{where $\by^*\in \ker(\bA)$ is assumed to be a steady state of  $\by'(t)=\bA\by(t)$. Note that the solution $\by$ satisfies $\by(t)\in H$ for all $t\geq 0$ if and only if $\by(0)\in H$.}
\begin{thm}[{\cite[Theorem 2.9]{izgin2022stability}}]\label{Thm_MPRK_stabil}
	Let $\bA\in \R^{N\times N}$ such that $\ker(\bA)=\Span(\bv_1,\dotsc,\bv_k)$ represents a $k$-dimensional subspace of $\R^N$ with $k\geq 1$. Also, let $\by^*\in \ker(\bA)$ be a fixed point of $\bg\from D\to D$ where $D\tm \R^N$ contains a neighborhood $\mathcal D$ of $\by^*$. Moreover, let any element of $\ker(\bA)\cap \mathcal D$ be a fixed point of $\bg$ and suppose that $\bg\big|_\mathcal{D}\in \mathcal C^1$ as well as that the first derivatives of $\bg$ are Lipschitz continuous on $\mathcal{D}$. Then $\bD\bg(\by^*)\bv_i=\bv_i$ for $i=1,\dotsc, k$ and the following statements hold.
	\begin{enumerate}
		\item\label{it:Thma} If the remaining $N-k$ eigenvalues of $\bD\bg(\by^*)$ have absolute values smaller than $1$, then $\by^*$ is stable.\label{It:Thm_Stab_a}
		\item\label{it:Thmb} Let $H$ be defined by \eqref{eq:H} and $\bg$ conserve all linear invariants, which means that $\bg(\by)\in H\cap D$ for all $\by\in H\cap D$. If additionally the assumption of \ref{It:Thm_Stab_a} is satisfied, then there exists a $\delta>0$ such that $\by^0\in H\cap D$ and $\norm{\by^0-\by^*}<\delta$ imply $\by^n\to \by^*$ as $n\to \infty$.
	\end{enumerate}
\end{thm}
We want to point out that the above theorem is a generalization of \cite[Theorem 2.9]{IKM2122} to linear systems of arbitrary finite size. \suggest{In a nutshell, this theorem allows to reduce the question of stability of a non-hyperbolic fixed point of a sufficiently smooth map to an eigenvalue problem, if the fixed points form a subspace. Even more, part \ref{it:Thmb} of this theorem allows us to conclude the local convergence of the iterates towards the steady state solution $\by^*$ along the plane $H$, if additionally all linear invariants are preserved. We additionally want to highlight, that these theorems can also be applied in the context of nonlinear systems of differential equations, see \cite[Remark 2.10]{izgin2022stability} or \cite{IKMnonlin22} for an example.}

\suggest{It is also worth mentioning, that using Theorem \ref{Thm:_Asym_und_Instabil} and Theorem \ref{Thm_MPRK_stabil} in the context of general linear methods such as Runge--Kutta schemes, the eigenvalues of the Jacobian are related to the known stability function of the method.}

\section{GeCo Schemes}
\label{sec:GeCo}
\suggest{A} class of numerical methods that preserve all linear invariants as well as positivity is given by GeCo schemes \suggest{introduced in} \cite{martiradonna2020geco}.
The first order GeCo scheme (GeCo1) applied to a general biochemical system (\ref{eq:bio})
is defined as
\begin{equation}
	\by^{n+1}=\by^n+\Delta t\varphi\left(\Delta t \sum_{j=1}^N  \dfrac{f_i^{[D]}(\by^n)}{y_i^n}\right)\bff(\by^n),\label{eq:GeCo1scheme}
\end{equation}
where the vector field $\bff(\by)=\bS\mathbf r(\by)$ is split into production and destruction parts as
\[\bff(\by)=\bff^{[P]}(\by)-\bff^{[D]}(\by),\qquad \bff^{[P]}(\by)=\bS^+\mathbf{r}(\by),\ \bff^{[D]}(\by)=\bS^-\mathbf{r}(\by),\quad \bS^+,\bS^-\geq \bzero,\]
and the function $\varphi$ is defined as
\begin{equation}\label{eq:phi}
	\varphi(x)=\begin{cases}
		\dfrac{1-e^{-x}}{x},& x>0,\\
		1, &x=0.
\end{cases}   \end{equation}
The second order GeCo scheme (GeCo2) for a general biochemical system (\ref{eq:bio}) is defined by
\begin{equation}\label{eq:GeCo2scheme}
\begin{aligned}
    \by^{(2)}&=\by^n+\Delta t\varphi\left(\Delta t \sum_{j=1}^N  \dfrac{f_i^{[D]}(\by^n)}{y_i^n}\right)\bff(\by^n),\\
	\mathbf y ^{n+1}&=\mathbf{y}^n + \dfrac{\Delta t}{2} \varphi\left(\Delta t \sum_{i=1}^N \dfrac{w_i^+(\mathbf y^n)}{y^n_i} \right) \left(\mathbf{f}(\mathbf y^n)+ \mathbf f(\mathbf y^{(2)})\right),
\end{aligned}
\end{equation}
where 
\begin{equation*}
	w_i^+(\by^n)=\max(0,w_i(\by^n)), \quad i=1,\dots,N
\end{equation*}
with 
\begin{equation*}
	\mathbf w(\by^n)=2\varphi\left(\Delta t \sum_{j=1}^N  \dfrac{f_i^{[D]}(\by^n)}{y_i^n}\right)\bff(\mathbf y^n) - \bff(\mathbf{y}^n)-\bff(\mathbf{y}^{(2)}).
\end{equation*}

We will start analyzing \suggest{GeCo1 applied to a general positive linear test problem with stable steady states and linear invariants. Turning to }GeCo2\suggest{, we} prove that already for the $2\times 2$ system \eqref{PDS_test} the stability domain of GeCo2 is bounded. 

\subsection{Stability of GeCo1}
In this section, we investigate the stability properties of GeCo1, see \eqref{eq:GeCo1scheme}, when applied to linear and positive systems of ordinary differential equations $\by'=\bA\by$. Before we formulate assumptions on the system matrix $\bA$, we introduce the  algebraic multiplicity $\mu_\bA(\lambda)$ of the eigenvalue $\lambda\in \sigma(\bA)$ as well as the corresponding geometric multiplicity $\gamma_\bA(\lambda)$, where $\sigma(\bA)$ is the spectrum of $\bA$. 

The presence of $k\geq 1$ linear invariants means that $\gamma_\bA(0)=k$, so that we consider in the following systems of the form
\begin{equation}\label{eq:PDS_Sys}
	\by'=\bA\by,\quad \bA\neq\bzero,\quad  \bA-\diag(\bA)\geq \bzero,\quad \mu_\bA(0)=\gamma_\bA(0)=k\geq 1,\quad \sigma(\bA)\tm \C^-=\{z\in \C\mid \re(z)\leq 0\},
\end{equation}
where $\diag(\bA)$ denotes the diagonal of $\bA$. In particular, $\bA-\diag(\bA)\geq \bzero$ means that $\bA$ is a Metzler matrix guaranteeing the positivity of the analytic solution. Moreover, in the presence of linear invariants, the conditions $\mu_\bA(0)=\gamma_\bA(0)$ and $\sigma(\bA)\tm \C^-$ are necessary for the stability of steady states of $\by'=\bA\by$, see \cite[Theorem 3.23]{MR1912409}. 

\begin{rem}\label{rem:Aneg}
	We want to mention that at least one diagonal element of $\bA$ is negative. Otherwise we find $\diag(\bA)\geq \bzero$, and hence, $\bA\geq \bzero$. Then, due to  $\mu_\bA(0)=\gamma_\bA(0)=k$ and $\bA\neq \bzero$ we find that $k<N$, and thus, there exists a nonzero eigenvalue of $\bA$. Therefore, $\bA$ is not similar to a strictly upper triangular matrix. Utilizing a generalization of the Perron-Frobenius Theorem \cite[Theorem 2.20]{V00} yields that $\bA$ possesses a positive eigenvalue contradicting $\sigma(\bA)\tm \C^-$. This means that $\bA$ is a so-called proper Metzler Matrix\suggest{, i.\,e.\ a Metzler matrix $\bA$ with at least one negative diagonal element}.
\end{rem}

For the analysis of GeCo1, we first rewrite $\by'=\bA\by$ as a production-destruction system of the form
\begin{equation}\label{eq:split_A}
	\by'=	\bA\by=\bS^+\by-\bS^-\by
\end{equation}
with $\bS^+,\bS^-\geq \bzero$. Since $\bA$ is a Metzler matrix, $\bS^-=(s^-_{ij})_{i,j=1,\dotsc,N}$ is a diagonal matrix and $f_i^{[D]}(\by)=s_{ii}^-y_i$. Moreover, Remark \ref{rem:Aneg} states that at least one diagonal element of $\bA$ is negative, which results in 
\begin{equation}
	\tr(\bS^-)>0.\label{eq:tr(S-)>0}
\end{equation}
With this in mind, let us recall the function $\varphi$ from \eqref{eq:phi}, that is
\begin{equation*}
	\varphi(x)=\begin{cases}\frac{1-e^{-x}}{x},& x>0,\\
		1, &x=0
	\end{cases}
\end{equation*} 
and write the GeCo1 scheme \eqref{eq:GeCo1scheme} applied to \eqref{eq:split_A} as
\begin{equation*}
	\begin{aligned}
		\mathbf{g}(\mathbf{y}^{n})&=\mathbf{y}^{n+1} =  \mathbf{y}^{n} + \Delta t \varphi\left(\Delta t\sum_{i=1}^N \dfrac{f_i^{[D]}(\by)}{y_i^{n}} \right) \bA\mathbf{y}^{n}=  \mathbf{y}^{n} + \Delta t\varphi(\Delta t\tr(\bS^-)) \bA\mathbf{y}^{n}.
\end{aligned}\end{equation*}
Due to \eqref{eq:tr(S-)>0}, the GeCo1 scheme can be rewritten as
\begin{equation}
	\by^{n+1}=\bg(\by^n)=(\bI+\Phi(\Delta t)\bA)\by^n,\quad \Phi(\Delta t)=\Delta t\varphi(\Delta t\tr(\bS^-))=\tfrac{1-e^{-\Delta t\tr(\bS^-)}}{\tr(\bS^-)}.\label{eq:ThmGeco1scheme}
\end{equation}

Note that steady states of \eqref{eq:PDS_Sys} become fixed points of $\bg$, and that $\bg\in \mathcal C^\infty$ conserves all linear invariants. Hence, we are in the position to apply Theorem \ref{Thm_MPRK_stabil}. It is worth noting that the eigenvalues of the Jacobian of the GeCo1 map $\bg$ in general not only depend on $\Delta t\lambda$, but also on the trace of $\bS^-$.
Nevertheless, we are able to prove that in the case of GeCo1, the remaining $N-k$ eigenvalues of $\bD\bg(\by^*)$ lie inside the unit circle, resulting in the following theorem.
\begin{thm}\label{Thm:GeCo1}
Any steady state 
	$\by^*$ of \eqref{eq:PDS_Sys} is a stable fixed point of GeCo1 given by \eqref{eq:ThmGeco1scheme} and there exists a $\delta>0$ such that $\bn_j^\mathsf{T} \by^0=\bn_j^\mathsf{T}\by^*$ for $\bn_j\in \ker(\bA^\mathsf{T})$ and $j=1,\dotsc,k$ as well as $\norm{ \by^0-\by^*}<\delta$ imply $\lim_{n\to\infty}\by^n=\by^*$ for all $\Delta t>0$. 
\end{thm}
\begin{proof}
	The Jacobian $\bD\bg(\by^*)$ reads
	\begin{equation}\label{eq:Dg(y*)Geco1}
		\bD\bg(\by^*)=\bI+\Phi(\Delta t)\bA=\bI+\frac{1-e^{-\Delta t\tr(\bS^-)}}{\tr(\bS^-)}\bA
	\end{equation} 
	and its eigenvalues are \[\mu=1+\Phi(\Delta t)\lambda\] with $\lambda\in \sigma(\bA)$.
	Hereby, we see that 
	$\bv\in \ker(\bA)\setminus\{ \bzero\}$ is an eigenvector of the Jacobian $\bD\bg(\by^*)$ with an associated eigenvalue of $1$.
	
	In order to investigate the location of the remaining $N-k$ eigenvalues of the Jacobian, 
	we first numerate the distinct and nonzero eigenvalues of $\bA$ from \eqref{eq:PDS_Sys} by $\lambda_1,\dotsc,\lambda_m$.
	Now, the corresponding eigenvalues $\mu_i=1+\Phi(\Delta t)\lambda_i$ with $i=1,\dotsc,m$ lie inside the unit circle if and only if
	\[\lvert 1+\Phi(\Delta t)\lambda_i \rvert^2<1,\quad i=1,\dotsc,m,\]
	which can be written as
	\[(\re(\Phi(\Delta t)\lambda_i)+1)^2+\im (\Phi(\Delta t)\lambda_i)^2<1,\quad i=1,\dotsc,m,\]
	or equivalently,
	\[2\Phi(\Delta t)\re(\lambda_i)+\Phi(\Delta t)^2\lvert \lambda_i\rvert^2<0,\quad  i=1,\dotsc,m.\]
	Dividing by $\Phi(\Delta t)>0$ and exploiting $\sigma(\bA)\tm \C^-$ gives
	\[\Phi(\Delta t)< -\frac{2\re(\lambda_k)}{\lvert \lambda_i\rvert^2}=\frac{2\lvert \re(\lambda_i)\rvert}{\lvert \lambda_i\rvert^2},\quad  i=1,\dotsc,m.\]
	\suggest{Introducing} 	
	\begin{equation}\label{eq:M_eig}
		M=\min_{i=1,\dotsc,m}\left\{2\tfrac{\lvert\re(\lambda_i)\rvert}{\lvert\lambda_i\rvert^2}\right\},
	\end{equation} we end up with the equivalent condition
	\[ \Phi(\Delta t)<M.\]
	Hence, after plugging in $\Phi(\Delta t)=\tfrac{1-e^{-\Delta t\tr(\bS^-)}}{\tr(\bS^-)}$,
	we multiply with its denominator $\tr(\bS^-)>0$, see \eqref{eq:tr(S-)>0}, to get
	\begin{equation*}
		\lvert \mu_i \rvert^2<1,\quad i=1,\dotsc,m
	\end{equation*}
	if and only if
	\begin{equation*}
	 1-e^{-\Delta t\tr(\bS^-)}<M \tr(\bS^-).
	\end{equation*}
	Now, if $M \tr(\bS^-)\geq 1$, then \[M\tr(\bS^-)\geq 1>1-e^{-\Delta t\tr(\bS^-)}\] is true for all $\Delta t>0$, and hence, the remaining eigenvalues of $\bD\bg(\by^*)$ associated with nonzero eigenvalues of $\bA$ lie inside the unit circle. We now aim to prove that 
	\[M \tr(\bS^-)\geq 1\]
	is indeed the case. 
	
Due to \cite[Theorem 10, Corollary 11]{StabMetz} it holds that \[\sigma(\bA)\tm\mathcal B= \left\{ z\in \C \, \Big| \, \abs{z-r}\leq \abs r, r=\min_{j=1,\dotsc N} a_{jj}\right\},\] \suggest{where $r<0$ follows from Remark \ref{rem:Aneg}. As a result, we find} $\re(\lambda)<0$ as well as $\arg(\lambda)\in(\tfrac\pi2,\tfrac32\pi)$  \suggest{for all} $0\neq\lambda\in\sigma(\bA)$. 
For \suggest{any} given $\lambda\in \sigma(\bA)\setminus\{0\}$, we define $\alpha=\pi-\arg(\lambda)\in(-\tfrac\pi2,\tfrac{\pi}{2})$, so that 
	\[ \cos(\alpha)=\frac{\abs{\re(\lambda)}}{\abs{\lambda}}\neq 0.\]
	Next, we choose $\theta<0$ satisfying
	\[\abs\lambda=\cos(\alpha)\abs\theta.\]
	A sketch of this geometry can be found in Figure \ref{Fig:proof}.
	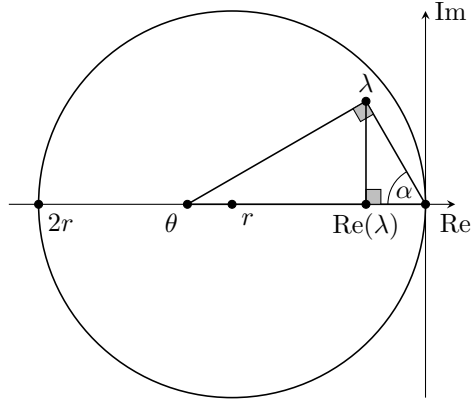
\begin{figure}[!h]
		\centering
		\begin{scaletikzpicturetowidth}{0.4\textwidth}
			\begin{tikzpicture}[scale=\tikzscale]
				\tkzInit[xmin=-5,xmax=0,ymin=-3,ymax=3]
				\draw [-stealth](-7,0) -- (0.5,0) node[below]{$\re$};
				\draw [-stealth](0,-3.25) -- (0,3.25) node[right]{$\im$};
				\tkzDefPoint[label=below left:{$\theta$}](-4,-0){P}
				\tkzDefPoint[label=below right:{$2 r $}](-6.5,-0){B}
				\tkzDefPoint[label=below right:{$ r $}](-3.25,-0){X}
				\tkzDefPoint(0,0){Q}
				\tkzDefPoint(-2,0){D}
				\tkzDefPoint[label=above:{$\lambda$}](-1,3^0.5){R}
				\tkzDrawSegments(P,Q Q,R R,P)
				\tkzDefPoint[label=below:{$\re(\lambda)$}](-1,0){N}

				\tkzMarkRightAngle[fill=lightgray](Q,N,R)
				\tkzMarkRightAngle[fill=lightgray](P,R,Q)
				\tkzDrawPoints[color=black](P,Q,R,N,B,X)
				\tkzDrawSegment(R,N)
				\pic [draw, -, "$\alpha$", angle eccentricity=0.65] {angle = R--Q--P};
				\tkzDrawCircles(X,B)
			\end{tikzpicture}
		\end{scaletikzpicturetowidth}
		\caption{Sketch of the geometric setup for $\arg(\lambda)\in(\tfrac\pi2,\pi)$ and $r=\min_{j=1,\dotsc,N}a_{jj}$.}\label{Fig:proof}
	\end{figure}
	With this, equation \eqref{eq:M_eig} becomes
	\[ M=\min_{i=1,\dotsc,m}\left\{2\tfrac{\lvert\re(\lambda_i)\rvert}{\lvert\lambda_i\rvert^2}\right\}=\min_{i=1,\dotsc,m}\left\{2\tfrac{\cos(\alpha_i)}{\lvert\lambda_i\rvert}\right\}=\min_{i=1,\dotsc,m}\left\{\tfrac{2}{\lvert \theta_i\rvert}\right\}.\]
	Moreover, with Thales's Theorem we can conclude that even $\theta_i\in \R^-$ is contained in $\mathcal B$, and thus, satisfies \[\abs{\theta_i}\leq 2\abs{\min_{j=1,\dotsc,N}a_{jj}}.\]
	Since $\bA$ is a proper Metzler matrix, see Remark \ref{rem:Aneg}, there exists an $l\in\{1,\dotsc,N\}$ such that \[\min_{j=1,\dotsc,N}a_{jj}=a_{ll}<0,\]
	from which it follows that
	\[M=\min_{i=1,\dotsc,m}\left\{\tfrac{2}{\lvert \theta_i\rvert}\right\}\geq\frac{2}{2\abs{\min_{j=1,\dotsc,N}a_{jj}}}=\frac{1}{\abs{a_{ll}}}. \]
	Additionally, setting $S=\{j\in\{1,\dotsc,N\}\mid a_{jj}<0\}$ we find
	\[\tr(\bS^-)=-\sum_{j\in S}^Na_{jj}=\sum_{j\in S}^N\lvert a_{jj}\rvert \geq \abs{a_{ll}},\]
	and thus
	\[M\tr(\bS^-)\geq\frac{1}{\abs{a_{ll}}}\lvert a_{ll}\rvert= 1,\]
	which finishes the proof.
\end{proof}
\suggest{With this theorem, a stability result for the GeCo1 scheme is provided for the first time. Having proved the unconditional stability of all fixed points of GeCo1 associated with steady states of the general $N\times N$ system of differential equations \eqref{eq:PDS_Sys}, we can conclude that GeCo1 mimics the stability behavior of the analytic solution close to a steady state solution for any chosen time step size $\Delta t>0$.}


\subsection{Stability of GeCo2}\label{subsec:StabGeCo}
In this section we aim to prove that GeCo2 applied to \eqref{PDS_test} can be described by a $\mathcal{C}^1$-map \suggest{using Lemma \ref{Lem:diff} from the appendix}, and to compute the spectrum of the corresponding Jacobian. To prove that the partial derivatives are even locally Lipschitz continuous, we use \suggest{Lemma \ref{Lem:locallyLip} from the appendix}.

\suggest{Let} us investigate the GeCo2 scheme applied to \eqref{PDS_test} with \[\bA=\underbrace{\Vec{0 & bc\\a & 0}}_{=\mathbf S^+} - \underbrace{\Vec{ac & 0\\0 & b}}_{=\mathbf S^-} \] and $\mathbf r(\by)=\by$. This means that $f^{[D]}(\by)=\mathbf S^-\mathbf r(\by)=\vec{acy_1\\by_2}$, and hence the GeCo2 scheme \suggest{\eqref{eq:GeCo2scheme}} reads
\begin{equation}\label{eq:yGeco2}
	\begin{split}
		\displaystyle \mathbf{y}^{(2)}&=\mathbf{y}^{n}+\Delta t\varphi(\Delta t\suggest{\tr(\bS^-)}) \bA\mathbf{y}^{n}\\
		\displaystyle \mathbf{y}^{n+1} &= 
		\mathbf{y}^{n} + \dfrac{1}{2} \Delta t\varphi\left(\Delta t\left(\dfrac{w_1^+(\by^n)}{y_1^n}+\dfrac{w_2^+(\by^n)}{y_2^n}\right)\right)  \bA\left(\mathbf{y}^{n} +\mathbf{y}^{(2)}\right)\\
		&=\mathbf{y}^{n} + \dfrac{1}{2} \Delta t\varphi\left(\Delta t\left(\dfrac{w_1^+(\by^n)}{y_1^n}+\dfrac{w_2^+(\by^n)}{y_2^n}\right)\right)  \bA\left(2\mathbf{y}^{n}+\Delta t\varphi(\Delta t\suggest{\tr(\bS^-)}) \bA\mathbf{y}^{n}\right),
	\end{split}
\end{equation}
where $	w_i^+(\by^n)=\max(0,w_i(\by^n))$ for $i=1,2$ and
\begin{equation}
	\begin{aligned}\label{eq:w_MatrixForm}
		\mathbf w(\by^n)&=2\varphi(\Delta t\suggest{\tr(\bS^-)})\bA\mathbf y^n - \bA\mathbf{y}^n-\bA\mathbf{y}^{(2)}\\&=\left(2 \varphi(\Delta t\suggest{\tr(\bS^-)})\bA-2\bA-\bA^2\Delta t\varphi(\Delta t\suggest{\tr(\bS^-)})\right)\by^n.
	\end{aligned}
\end{equation}

We formulate a helpful lemma to understand some properties of $\bw$ and to express equation \eqref{eq:yGeco2} with $\bw$ rather than $w_1^+$ and $w_2^+$.
\begin{lem}\label{Lem:w}
	The map $\bw$ from \eqref{eq:w_MatrixForm} with $\bA$ from \eqref{PDS_test} satisfies $w_1=-\tfrac1cw_2$, and we have
	\[w_1(\by^n)\begin{cases}>0, &y^n_1>\tfrac{b}{a}y^n_2, \\
		=0, &y^n_1=\tfrac{b}{a}y^n_2,\\
		<0, & y^n_1<\tfrac{b}{a}y^n_2.\end{cases} \]
\end{lem}
\begin{proof}
	First note that $y^n_1=\tfrac{b}{a}y^n_2$ is equivalent to $\by^n\in \ker(\bA)$, and thus \eqref{eq:w_MatrixForm} yields $\bw(\by^n)=\bzero$.
	
	Next, we focus on finding conditions for $\by^n$ so that $w_1(\by^n)>0$. For this, it is worth mentioning that for every $\by^n>\bzero$\suggest{,} there exists a unique $\by^*\in \ker(\bA)\cap \R^2_{>0}$ satisfying  $\bn^\mathsf{T}\by^n=\bn^\mathsf{T}\by^*$ with $\bn=(1,c)^\mathsf{T}$. Hence, \suggest{since $\by^*>\bzero$ and $\bby=(1,-\tfrac1c)^\mathsf{T}$ are linearly independent,} there exists a unique $s^n\in\R$ such that $\by^n=\by^*+s^n\bby$. \suggest{Also note that $\bA\bby=\lambda\bby$ with $\lambda=-(ac+b)<0$ and} 
	\begin{equation}\label{eq:sn}
		s^n\begin{cases}
			>0, &y^n_1>\tfrac{b}{a}y^n_2,\\
			=0, &y^n_1=\tfrac{b}{a}y^n_2,\\
			<0, &y^n_1<\tfrac{b}{a}y^n_2.
		\end{cases}
	\end{equation}
	Thus, the linearity of $\bw$ and \eqref{eq:phi} lead to
	\begin{equation}\label{eq:w} 
		\begin{aligned}
			\bw(\by^n)&=\bw(\by^*)+\bw(s^n\bby)=\left(2 \varphi(-\Delta t\lambda)\lambda-2\lambda-\lambda^2\Delta t\varphi(-\Delta t\lambda)\right)s^n\bby\\
			&=\frac{1}{\Delta t}\left(2(1-e^{\Delta t\lambda}) -2\Delta t\lambda-\Delta t\lambda(1-e^{\Delta t\lambda})\right)s^n\bby\\
			&=\frac{1}{\Delta t}\left(2 -3\Delta t\lambda+ e^{\Delta t\lambda}(\Delta t\lambda-2)\right)s^n\bby.
		\end{aligned}
	\end{equation}
	Furthermore, introducing the function
	\[p(z)=-3z+2-e^z(2-z),\]
	we can rewrite \eqref{eq:w} to get
	\begin{equation}\label{eq:w_final}
		\bw(\by^n)=\frac{1}{\Delta t}p(\Delta t\lambda)s^n\bby.
	\end{equation}
	Now, the first derivative of $p$ satisfies
	\[p'(z)=-(3+e^z(1-z))<0\]
	for all $z\leq0.$ Hence, the function $p$ is strictly decreasing for $z\leq 0$ and satisfies $p(0)=0$ proving that $p(\lambda\Delta t)>0$ for all $\Delta t>0$. 
	Therefore, with \eqref{eq:sn} it follows that $w_1(\by^n)>0$ if $y_1^n>\frac{b}{a}y_2^n$. Similarly,  $w_1(\by^n)<0$ holds if $y_1^n<\frac{b}{a}y_2^n$. Finally, note that  \eqref{eq:w_final} implies $w_1(\by^n)=-\tfrac1cw_2(\by^n)$.
\end{proof} 

As a consequence of this lemma we simplify \eqref{eq:yGeco2} by introducing $H\colon \R^2_{>0}\to \R_{>0}$ with 
\begin{equation}\label{eq:psi_w1}
	\begin{aligned}
		H(\bx)&=\Delta t \varphi\left(\Delta t\left(\dfrac{w_1^+(\bx)}{x_1}+\dfrac{w_2^+(\bx)}{x_2}\right)\right)  =\begin{cases}\widetilde H_1(\bx),& x_1>\frac{b}{a} x_2,\\
			\Delta t,& x_1= \frac{b}{a} x_2,\\
			\widetilde H_2(\bx),& x_1< \frac{b}{a} x_2,
		\end{cases}\\
		\widetilde H_i(\bx)&=\dfrac{1-e^{-\Delta t \tfrac{w_i(\bx)}{x_i}}}{\tfrac{w_i(\bx)}{x_i}},\quad i=1,2
	\end{aligned}
\end{equation}
and point out that $H$ is continuous, since $\varphi$ from \eqref{eq:phi} is in $ \mathcal C^1$ and $\bw\in \mathcal C^\infty$. As a result of Lemma \ref{Lem:w}, we even know that \[\widetilde H_i\in \mathcal C^\infty(\R^2_{>0}\setminus\ker(\bA))\] for $i=1,2$.

The map $\bg$ defining the iterates of the GeCo2 scheme when applied to \eqref{PDS_test} is given by \eqref{eq:yGeco2} and can be written as
\begin{equation*}
	\begin{aligned}
		\bg(\bx)&=\bx+ \dfrac{1}{2} H(\bx)  \bA\left(2\bx+\Delta t\varphi(\Delta \suggest{\tr(\bS^-)})\bA \bx\right).
	\end{aligned}
\end{equation*}
Introducing $\bG(\bx)=\bA\bx H(\bx)$ we obtain
\begin{equation}\label{eq:gGeCo2}
	\bg(\bx)=\bx+\bG(\bx)+\frac12\Delta t\varphi(\Delta t\suggest{\tr(\bS^-)})\bA\bG(\bx).
\end{equation}
The following theorem uses this representation of $\bg$ to analyze the stability properties of GeCo2.
\begin{thm}\label{Thm:geco2}
	Let $\bg$, given by \eqref{eq:gGeCo2}, be the generating function of the GeCo2 iterates $\by^n$ when applied to \eqref{PDS_test}. Further, let $\by^*>\bzero$ be a steady state solution of \eqref{PDS_test}. 
	Then the map $\bg$ is in $\mathcal C^1(\mathcal D)$ and has Lipschitz continuous derivatives on a sufficiently small neighborhood $\mathcal D$ of $\by^*$.
	Moreover, the stability function of the GeCo2 scheme reads
	\begin{equation}\label{eq:stabfun_GeCo2}
		R(z)=\suggest{1+z+\frac12 z^2\varphi(\Delta t\tr(\bS^-))}.
	\end{equation}
	If $\lvert R(-\Delta t(ac+b))\rvert <1$, then $\by^*$ is stable and there exists a $\delta>0$ such that $(1,c)\by^0=(1,c)\by^*$ and $\norm{\by^0-\by^*}<\delta$ imply $\lim_{n\to \infty}\by^n=\by^*$.  If $\lvert R(-\Delta t(ac+b))\rvert >1$, then $\by^*$ is an unstable fixed point of GeCo2.
\end{thm}
\begin{proof}
	We demonstrate that all assumptions of Theorem \ref{Thm_MPRK_stabil} and Theorem \ref{Thm:_Asym_und_Instabil} are fulfilled.
	
	From part \ref{it:alemdiff} of Lemma \ref{Lem:diff} \suggest{from the appendix} with $\mathbf \Phi(\bx)=\bA\bx$ and $\Psi(\bx)=H(\bx)$, it follows that the partial derivatives of $\bG(\bx)=\bA\bx H(\bx)$ on $\ker(\bA)$ exist and that $\bD\bG(\bx_0)=\Psi(\bx_0)\bA=\Delta t\bA$ holds for $i=1,2$ and all $\bx_0\in \ker(\bA)$.
	As a result of \eqref{eq:gGeCo2} we obtain
	\begin{equation}\label{eq:Dg(y*)_Geco2}
		\bD\bg(\by^*)=\bI+\Delta t\bA+\frac{1}{2}(\Delta t)^2\varphi(\Delta t\suggest{\tr(\bS^-)})\bA^2
	\end{equation}
	and the eigenvalues are given by $1$ and $R(-\Delta t(ac+b))$, where
	\begin{equation*}
		R(z)=1+z+\frac12 z^2\varphi(-\suggest{\Delta t\tr(\bS^-)}).
	\end{equation*}
	In total, we can write
	\begin{equation}\label{eq:DGGeCo2}
		\bD\bG(\bx)=\bB(\bx)+\bC(\bx)
	\end{equation}
	with
	\[\bB(\bx)=\bA H(\bx)\qta \bC(\bx)=\bA\bx\cdot\begin{cases}\nabla \widetilde H_1(\bx),& x_1>\frac{b}{a}x_2,\\
		\bzero^\mathsf{T} ,& x_1=\frac{b}{a}x_2,\\
		\nabla \widetilde H_2(\bx),& x_1<\frac{b}{a}x_2. \end{cases} \] 
	Note that, if each entry of $\bB\suggest{=(b_{ij})_{i,j=1,2}}$ and $\bC\suggest{=(c_{ij})_{i,j=1,2}}$ satisfies the assumptions of Lemma \ref{Lem:locallyLip} \suggest{from the appendix}, we can conclude that $\bG\in \mathcal C^1(\mathcal D)$ in a sufficiently small neighborhood $\mathcal D$ of $\by^*$ and that the first derivatives are Lipschitz continuous on $\mathcal{D}$. As a direct consequence of \eqref{eq:gGeCo2}, the same would then hold true for $\bg$.
	
	Now we show that the entries \suggest{$b_{ij}$ and $c_{ij}$ of of the matrices} $\bB$ and $\bC$ satisfy the assumptions of Lemma \ref{Lem:locallyLip} \suggest{from the appendix}, that is 
	\begin{enumerate}
		\item $\suggest{b_{ij}}$ and  $\suggest{c_{ij}}$ are continuous on $\R^2_{>0}$,
		\item  $\suggest{b_{ij}}$ and  $\suggest{c_{ij}}$ are constant on $\ker(\bA)$,
		\item  $\suggest{b_{ij}}$ and  $\suggest{c_{ij}}$ are in $\mathcal C^1$ on $\R^2_{>0}\setminus\ker(\bA)$ and
		\item $\lim_{\bx\to\bx_0}\nabla \suggest{b_{ij}(\bx)}$ as well as $\lim_{\bx\to\bx_0}\nabla \suggest{c_{ij}(\bx)}$ exist for all $\bx_0\in \ker(\bA)\cap \R^2_{>0}$
	\end{enumerate}
	for $i,j\in\{1,2\}$. First, note that $\bB$ and $\bC$ are constant on $\ker(\bA)$, and due to $\widetilde H_k\in \mathcal C^2$ for $k=1,2$\suggest{,}  we find that each entry of the two matrices is continuously differentiable on $\R^2_{>0}\setminus \ker(\bA)$. Even more, since $H$ is continuous we know that $\suggest{b_{ij}}\in \mathcal C(\R^2_{>0})$ for  $i,j\in\{1,2\}$. 
	
	We want to point out that if $\lim_{\bx\to\bx_0}\nabla \widetilde H_k(\bx)$ exists, this proves the continuity of $\suggest{c_{ij}}$ as well as that $\lim_{\bx\to\bx_0}\nabla\suggest{b_{ij}(\bx)}$ exists for all $i,j\in \{1,2\}$.
	Furthermore, for $\bx\notin \ker(\bA)$ we find
	\[\nabla\suggest{c_{ij}(\bx)}= \nabla(\bA\bx\nabla \widetilde H_k(\bx))_{ij}=\nabla((\bA\bx)_i\tfrac{\partial}{\partial x_j}\widetilde H_k(\bx))=(\bA\be_i)^\mathsf{T} \tfrac{\partial}{\partial x_j}\widetilde H_k(\bx)+(\bA\bx)_i \nabla (\tfrac{\partial}{\partial x_j}\widetilde H_k(\bx))\]
	and see that $\lim_{\bx\to\bx_0}\nabla \suggest{c_{ij}(\bx)}$ exists if  $\lim_{\bx\to\bx_0}\nabla \widetilde H_k(\bx)$ as well as $\lim_{\bx\to\bx_0}\nabla (\tfrac{\partial}{\partial x_j}\widetilde H_k(\bx))$ exist for $i,j,k\in\{1,2\}$. 
	To see that both limits exist for $\bx_0\in \ker(\bA)\cap \R^2_{>0}$, we introduce $\Phi(z)=\frac{1-e^{-\Delta tz}}{z}$, so that $ \widetilde H_k(\bx)=\Phi(\tfrac{w_k(\bx)}{x_k})$. 
	Hence, we have  $\Phi\in \mathcal C^2(\R\setminus\{0\})$ and
	\begin{equation}
		\begin{aligned}\label{eq:NablaH_i(x)}
			\nabla\widetilde H_k(\bx)=&\Phi'(\tfrac{w_k(\bx)}{x_k})\left(\frac{\nabla w_k(\bx)}{x_k}+w_k(\bx)\nabla\left(\frac{1}{x_k}\right) \right),\\
			\nabla(\tfrac{\partial}{\partial x_j}\widetilde H_k(\bx)) =&\Phi''(\tfrac{w_k(\bx)}{x_k})\left(\frac{\nabla w_k(\bx)}{x_k}+w_k(\bx)\nabla \left(\frac{1}{x_k}\right) \right)\left(\frac{\tfrac{\partial}{\partial x_j} w_k(\bx)}{x_k}+w_k(\bx)\tfrac{\partial}{\partial x_j}\left(\frac{1}{x_k}\right) \right)\\
			&+\Phi'(\tfrac{w_k(\bx)}{x_k})\nabla\left(\frac{\tfrac{\partial}{\partial x_j} w_k(\bx)}{x_k}+w_k(\bx)\tfrac{\partial}{\partial x_j}\left(\frac{1}{x_k}\right) \right).
		\end{aligned}
	\end{equation}
	As $\lim_{\bx\to\bx_0}\tfrac{w_k(\bx)}{x_k}=0$ for $\bx_0\in \ker(\bA)\cap \R^2_{>0}$, see \eqref{eq:w_MatrixForm}, we are interested in the limits of the first two derivatives of $\Phi$ at $z=0$. By l'Hospital's rule, a straightforward calculation yields
	\begin{equation}
		\lim_{z\to 0}\Phi'(z)=-\frac{(\Delta t)^2}{2} \qta \lim_{z\to 0}\Phi''(z)=\frac{(\Delta t)^3}{3}.\label{eq:limitPhi'}
	\end{equation}
	In addition, due to \eqref{eq:w_MatrixForm}, we know that $\nabla w_k$ is a constant function for $\suggest{k}=1,2$, which means that \[\nabla\left(\frac{\frac{\partial}{\partial x_j} w_k(\bx)}{x_k}+w_k(\bx)\frac{\partial}{\partial x_j}\left(\frac{1}{x_k}\right) \right)=\frac{\partial}{\partial x_j} w_k(\bx)\nabla\left(\frac{1}{x_k}\right)+\nabla w_k(\bx)\frac{\partial}{\partial x_j}\left(\frac{1}{x_k}\right)+ w_k(\bx)\nabla\left(\frac{\partial}{\partial x_j}\left(\frac{1}{x_k}\right)\right).\]
	It thus follows from \eqref{eq:NablaH_i(x)} and \eqref{eq:limitPhi'} that  $\lim_{\bx\to\bx_0}\nabla \widetilde H_k(\bx)$ as well as $\lim_{\bx\to\bx_0}\nabla (\tfrac{\partial}{\partial x_j}\widetilde H_k(\bx))$ exist for all $i,j\in\{1,2\}$ and each $\bx_0\in \ker(\bA)\cap \R^2_{>0}$. 
\end{proof}
\begin{rem}\label{rem:geco2}
	A numerical calculation shows that the stability function $R$ from \eqref{eq:stabfun_GeCo2} \suggest{with $\Delta t\tr(\bS^-)=-z$} satisfies $\abs{R(z)}<1$ for $z\in (z^*,0]$ with $-3.9924\leq z^*\leq -3.9923$.  Hence, the stability region of GeCo2 \suggest{when applied to \eqref{PDS_test}} is almost twice as big as the one of the underlying Heun scheme which is $(-2,0]$. 
\end{rem}

\section{Generalized BBKS Schemes}
\label{sec:BBKS}
The gBBKS schemes were developed in \cite{BBKS2007,BRBM2008, MR4109346} and represent a class of schemes that are unconditionally positive while preserving all linear invariants of the underlying ordinary differential equation $\by'=\bff(t,\by)\in \R^N$. The first order gBBKS schemes (gBBKS1) can be written as
\begin{equation}\label{eq:gBBKS1Intro}
	y^{n+1}_i=y^n_i + \Delta t f_i(t,\by^n) \Bigg(\prod_{m\in M^{n}} \frac{y^{n+1}_m}{\sigma^{n}_m}\Bigg)^{\ns r^{n}},\quad i=1,\dotsc,N,
\end{equation}
where $r^n,\sigma_m^n>0$ are \suggest{free parameters, but need to be chosen} independent\suggest{ly} of $\by^{n+1}$ and
\begin{equation*}
	M^{n}=\{ m\in \{1,\dotsc,N\} \mid  f_m(t,\by^n)<0\}.
\end{equation*}
\suggest{For instance, the BBKS1 scheme from \cite{BBKS2007,MR4109346} is given by setting $\sigma^n_m=y^n_m$ and $r^n=1$.} The second order gBBKS schemes (gBBKS2($\alpha$)) read
\
\begin{equation}\label{eq:gBBKS2Intro}
	\begin{aligned}
		\mathllap{y_i^{(2)}} &= y_i^n+\alpha \Delta t f_i(t^n,\by^{n}) \Bigg(\prod_{j\in J^{n}} \frac{y^{(2)}_j}{\pi^{n}_j}\Bigg)^{\ns q^{n}},\\
		\mathllap{y_i^{n+1}} &= y_i^n+\Delta t \left(\Big(1-\frac{1}{2\alpha}\Big) f_i(t^n,\by^{n})+ \frac{1}{2\alpha}f_i(t^{n}+\alpha \Delta t,\by^{(2)})\right)\Bigg(\prod_{m\in M^{n}} \frac{y^{n+1}_m}{\sigma^{n}_m}\Bigg)^{\ns r^{n}},
	\end{aligned}
\end{equation}
for $i=1,\dots,N$ with $\alpha\geq \frac12$ and $\pi_j^n,q^n>0$ being \suggest{free parameters chosen} independent\suggest{ly} of $\by^{(2)}$ while $\sigma_m^n,r^n>0$ are \suggest{chosen to be} independent of $\by^{n+1}$. \suggest{To give an example, the BBKS2(1) scheme from \cite{BRBM2008, MR4109346} uses $\pi^n_m=\sigma^n_m=y^n_m$ and $q^n=r^n=1$.} Moreover, the sets $J^n$ and $M^n$ are given by
\begin{subequations}
	\begin{align*}
		J^{n}&=\left\{ j\in \{1,\ldots,N\}\mid f_j(t^n, \by^{n})<0\right\},\\
		M^{n}&=\left\{ m\in \{1,\ldots,N\}\;\Big |\; \Big(1-\frac{1}{2\alpha}\Big) f_m(t^n,\by^{n})+ \frac{1}{2\alpha}f_m(t^{n}+\alpha \Delta t,\by^{(2)})<0\right\}.
	\end{align*}
\end{subequations}
\suggest{We want to note that $M^n$ always refers to the last step of the corresponding method.}



\subsection{Stability of first order gBBKS Schemes}\label{subsec:StabBBKS}
When applied to the system of differential equations \eqref{PDS_test}, i.\,e.\ $\by'=\bA\by$, the first order gBBKS schemes \eqref{eq:gBBKS1Intro} are given by 
\begin{equation}\label{eq:gBBKS11}
	y^{n+1}_i=y^n_i + \Delta t (\bA\by^n)_i \Bigg(\prod_{m\in M^{n}} \frac{y^{n+1}_m}{\sigma^{n}_m}\Bigg)^{\ns r^{n}},\quad i=1,2,
\end{equation}
where
\begin{equation}
	M^{n}=\{ m\in \{1,2\}\mid  (\bA\by^n)_m<0\}.\label{eq:defM}
\end{equation}

In this section we investigate the stability properties of gBBKS schemes by first proving that the assumptions of Theorem \ref{Thm_MPRK_stabil} are met.
The existence and uniqueness of a function $\bg$ generating the iterates from \eqref{eq:gBBKS11}, i.\,e.\ $\by^{n+1}=\bg(\by^n)$, is \suggest{already} proven in \cite{MR4109346}. Thereby, $\bg$ is given by the unique solution to some equation 
\[\bF(\bx,\bg(\bx))=\bzero, \]
where $\bF\colon \R^2_{>0}\times  \R^2_{>0}\to \R^2$ with $(\bx,\by)\mapsto\bF(\bx,\by)$.  In the following we denote by 
\begin{equation*}
	\begin{aligned}
		\bD_\bx\bF(\bx,\by)&=\frac{\partial \bF}{\partial \bx}(\bx,\by),\\
		\bD_\by\bF(\bx,\by)&=\frac{\partial \bF}{\partial \by}(\bx,\by)
	\end{aligned}
\end{equation*}
the Jacobians of $\bF$ with respect to $\bx$ and $\by$, respectively.

An intuitive way of proving $\bg\in \mathcal C^1(\mathcal D)$, where $\mathcal D$ is a neighborhood of a fixed point $\by^*$ of $\bg$, is to use the implicit function theorem. Unfortunately, we will see in the following that in our case $\bF$ is not differentiable on $\mathcal D\times \mathcal D$. Since the existence and uniqueness of the map $\bg$ is already known here, the differentiability of $\bg$ can be obtained by weaker assumptions on $\bF$ as the next theorem states.
\begin{thm}[{\cite[Theorem 11.1]{LS14}}]\label{Thm:gDiff}
	Let $D\tm \R^2$ be open and $\bg\colon D\to D$ be continuous in $\bx_0$. Furthermore, let $\bF\colon D\times D\to \R^2$ with $(\bx,\by)\mapsto\bF(\bx,\by)$ be differentiable in $(\bx_0,\bg(\bx_0))^\mathsf{T}$ and $\bD_\by\bF(\bx_0,\bg(\bx_0))$ be invertible. Suppose that $\bF(\bx,\bg(\bx))=\bzero$ for all $\bx\in D$, then also $\bg$ is differentiable in $\bx_0$ and 
	\begin{equation*}
		\bD\bg(\bx_0)=-(\bD_\by\bF(\bx_0,\bg(\bx_0)))^{-1}\bD_\bx\bF(\bx_0,\bg(\bx_0)).
	\end{equation*} 
\end{thm}

%
%

Before we formulate the stability theorem for \suggest{gBBKS1}, we introduce some assumptions on the exponent $r^n$ as well as $\sigma_m^n$ from \eqref{eq:gBBKS11}.
In particular, $r^n>0$ and $\sigma_m^n>0$ may depend on $\by^n$ and hence will be interpreted as functions $r^n=r(\by^n)$ and $\sigma_m^n=\sigma_m(\by^n)$. For the analysis of the gBBKS1 schemes we do not \suggest{further} specify the expressions for \suggest{the functions} $r$ or $\sigma_m$. Instead, we assume some reasonable properties such as that $r$ and $\sigma_m$ are positive for all $\Delta t>0$. Furthermore, we require $\sigma_m(\bv)=v_m$ whenever $\bv\in\ker(\bA)\cap \R^2_{>0}$ which is in agreement with the literature \cite{MR4109346,BBKS2007,BRBM2008}. To guarantee the regularity of the map generating the iterates $\by^n$, we also assume that $r,\sigma_1$ and $\sigma_2$ are in $\mathcal C^2$. In total, we prove the following theorem.

\begin{thm}\label{Thm:gBBKS1}
	Let  $\by^*>\bzero$ be a steady state solution of \eqref{PDS_test}, and $\sigma_1,\sigma_2,r\in \mathcal{C}^2(\R^2_{>0},\R_{>0})$. Further, let $\mathcal D$ be a sufficiently small neighborhood of $\by^*$ and suppose that $\bm \sigma(\bv)=\bv$ for all $\bv\in C=\ker(\bA)\cap \mathcal D$. Then the map $\bg$ generating the iterates of the gBBKS1 family\suggest{,} implicitly given by \eqref{eq:gBBKS11}\suggest{,} satisfies $\bg(\bv)=\bv$ for all steady states $\bv\in  C$ and the following statements hold.
	\begin{enumerate}
		\item\label{it:aThmgbbks1} The map $\bg$ satisfies  $\bg\in \mathcal{C}^1(\mathcal D)$ and $\bD\bg(\by^*)=\bI+\Delta t\bA$.
		\item\label{it:cThmgbbks1}  The first derivatives of $\bg$ are bounded and Lipschitz continuous on $\mathcal D$.
	\end{enumerate}
\end{thm}
\begin{proof}
	Before we start the proof of \ref{it:aThmgbbks1}, we make some preparatory considerations.
	
	Since $(\bA\by^n)_1=c(-ay_1^n+by_2^n)$ and $c(\bA\by^n)_2=-(\bA\by^n)_1$ we find
	\begin{equation}\label{eq:MgBBKS1}
		M^n=\begin{cases}
			\{1\}, & y_1^n>\frac{b}{a}y^n_2,\\
			\emptyset, & y_1^n=\frac{b}{a}y^n_2,\\
			\{2\}, & y_1^n<\frac{b}{a}y^n_2.
		\end{cases}
	\end{equation}
	Hence, when applied to \eqref{PDS_test}, \eqref{eq:IC} the scheme \eqref{eq:gBBKS11} turns into
	\begin{equation}\label{eq:gBBKS11_Test_Prob}
		\by^{n+1}=\by^n + \Delta t \bA \by^n\begin{cases}
			\left(\frac{y^{n+1}_1}{\sigma^{n}_1}\right)^{r^n}, & y_1^n>\frac{b}{a}y^n_2,\\
			1, & y_1^n=\frac{b}{a}y^n_2,\\
			\left(\frac{y^{n+1}_2}{\sigma^{n}_2}\right)^{r^n}, & y_1^n<\frac{b}{a}y^n_2,
		\end{cases}
	\end{equation}
	where $(\frac{b}{a}y_2^n,y_2^n)^\mathsf{T}\in \ker(\bA)$ is a steady state solution of \eqref{PDS_test}.
	
	Recall that the map $\bg$ generates the iterates $\by^n$, that is $\by^{n+1}=\bg(\by^n)$. Hence, inserting $\by^n=\bv\in C$ into equation \eqref{eq:gBBKS11_Test_Prob} yields $\by^{n+1}=\bg(\bv)$ on the left and $\bv$ on the right, and thus $\bg(\bv)=\bv$. Furthermore, we introduce the function $\bF$ defined by
	\begin{equation}\label{eq:F(x).gBBKS1}
		\begin{aligned}
			\bF&\colon \R^2_{>0}\times \R^2_{>0}\to \R^2,\\ \bF(\bx,\by)&=\by-\bx-\Delta t\bA\bx H(\bx,\by),\\
			H(\bx,\by)&=\begin{cases}
				\widetilde H_1(\bx,\by), & x_1>\frac{b}{a}x_2,\\
				1, & x_1=\frac{b}{a}x_2,\\
				\widetilde H_2(\bx,\by), & x_1<\frac{b}{a}x_2,
			\end{cases} \\
			\widetilde H_i(\bx,\by)&=\left(\frac{y_i}{\sigma_i(\bx)}\right)^{r(\bx)}, \quad i=1,2,
		\end{aligned}
	\end{equation}
	which satisfies $\bF(\bx,\bg(\bx))=\bzero$ for all $\bx>\bzero$.

	\begin{enumerate}
		\item We first show that $\bF$ is not differentiable on $\mathcal D\times \mathcal D$. For this, we choose $\bx_0\in C$ as well as $\by_0>\bzero$ with $\frac{(\by_0)_1}{(\bx_0)_1}\neq\frac{(\by_0)_2}{(\bx_0)_2}$ and define $\Psi(\bx)=H(\bx,\by_0)$. As a result of $\bm \sigma(\bx_0)=\bx_0$ it follows that \[\lim_{h\searrow 0}\Psi(\bx_0+h\be_1)=\lim_{h\searrow 0} \widetilde{H}_1(\bx_0+h\be_1,\by_0)=\left( \frac{(\by_0)_1}{\sigma_1(\bx_0)}\right)^{r(\bx_0)}=\left( \frac{(\by_0)_1}{(\bx_0)_1}\right)^{r(\bx_0)}.\]
		Analogously, we obtain
		\[\lim_{h\nearrow 0}\Psi(\bx_0+h\be_1)=\lim_{h\searrow 0} \widetilde{H}_2(\bx_0+h\be_1,\by_0)=\left( \frac{(\by_0)_2}{\sigma_2(\bx_0)}\right)^{r(\bx_0)}=\left( \frac{(\by_0)_2}{(\bx_0)_2}\right)^{r(\bx_0)},\]
		which  shows that $\Psi(\bx_0+h\be_1)$ possesses several accumulation points as $h\to 0$, and hence, part \ref{it:blemdiff} of Lemma \ref{Lem:diff} \suggest{from the appendix} with $\mathbf \Phi(\bx)=\bA\bx$ implies that the $1$st partial derivative of $\bF$ does not exist.
		
		As mentioned above\suggest{,} this means that we can not apply the implicit function theorem to $\bF$ on $\mathcal D\times \mathcal D$ in order to prove that $\bg\in \mathcal C^1(\mathcal D)$. Nevertheless, $\bF$ is differentiable in $(\bx,\by)\in E= \mathcal D\setminus\ker(\bA)\times \mathcal D$, since in this case we have
		\begin{equation}\label{eq:F(x,y):i=1,2}
			\bF(\bx,\by)=\by-\bx-\Delta t\bA\bx\left(\frac{y_i}{\sigma_i(\bx)}\right)^{r(\bx)}, \quad i=\begin{cases}
				1,& x_1>\frac{b}{a}x_2,\\
				2,& x_1<\frac{b}{a}x_2
			\end{cases} 
		\end{equation}
		with $\sigma_1,\sigma_2,r\in \mathcal{C}^2(\R^2_{>0},\R_{>0})$.
		In order to show that $\bg\in \mathcal C^1(\mathcal D\setminus \ker(\bA))$, we first show that the inverse of $\bD_\by\bF(\bx,\bg(\bx))$ exists for all $\bx\in \mathcal D\setminus \ker(\bA)$. It is straightforward to verify that 
		\begin{equation*}
			\bD_\by\bF(\bx,\bg(\bx))=\bI-\Delta t\bA \bx\nabla_\by\widetilde H_i(\bx,\bg(\bx))=\bI-\Delta t\bA \bx\be_i^\mathsf{T}\frac{r(\bx)}{\sigma_i(\bx)}\left(\frac{g_i(\bx)}{\sigma_i(\bx)}\right)^{r(\bx)-1}
		\end{equation*}
		holds for $\bx\notin C$.
		Introducing the vectors
		\begin{equation}
			\bv^{(i)}(\bx)=\Delta t \bA\bx \frac{r(\bx)}{\sigma_i(\bx)}\left(\frac{g_i(\bx)}{\sigma_i(\bx)}\right)^{r(\bx)-1}=\Delta t \bA\bx \frac{r(\bx)}{g_i(\bx)}\left(\frac{g_i(\bx)}{\sigma_i(\bx)}\right)^{r(\bx)}\label{eq:v^(i)}
		\end{equation}
		for $i$ from \eqref{eq:F(x,y):i=1,2}, we can write the Jacobian in the compact form
		\begin{equation}\label{eq:DyFtilde=I-ve_i}
			\bD_\by\bF(\bx,\bg(\bx))=\bI- \bv^{(i)}(\bx)\be_i^\mathsf{T}.
		\end{equation}
		Note that due to \eqref{eq:DyFtilde=I-ve_i},  the Jacobian of $\bF$ with respect to $\by$ is a triangular matrix, depending on $i$ from \eqref{eq:F(x,y):i=1,2}. Nevertheless, in either case  we find
		\begin{equation}\label{eq:det()=1-v}
			\det(\bD_\by\bF(\bx,\bg(\bx)))=1-v_i^{(i)}(\bx).
		\end{equation}
		Now, we know that $(\bA\bx)_i< 0$ for $i$ form \eqref{eq:F(x,y):i=1,2} by construction of the gBBKS schemes, which in particular means that
		\begin{equation}\label{eq:v_ineq1}
			v^{(i)}_i(\bx)\neq 1.
		\end{equation} As a result \suggest{of \eqref{eq:det()=1-v}}, the inverse of $\bD_\by\bF(\bx,\bg(\bx))$ exists.
		
		Considering a zero $(\bx_0,\bg(\bx_0))\in E$ of $\bF$, the implicit function theorem thus provides the existence of a unique $\mathcal C^1$-map $\widetilde{\bg}$ satisfying $\bF(\bx,\widetilde{\bg}(\bx))=\bzero$ in a sufficiently small neighborhood of $(\bx_0,\bg(\bx_0))$. Since $\bg$ and $\widetilde{\bg}$ are unique, we find $\bg=\widetilde{\bg}$, and since $\bx_0$ was arbitrary, we have shown that $\bg\in \mathcal C^1$ on $\mathcal D\setminus \ker(\bA)$, and in particular
		\begin{equation}\label{eq:Dg(x)}
			\bD\bg(\bx)=-(\bD_\by\bF(\bx,\bg(\bx)))^{-1}\bD_\bx\bF(\bx,\bg(\bx))
		\end{equation}
		for $\bx\in\mathcal D\setminus \ker(\bA)$. It thus remains to show that $\bg\colon D\to D$ is also differentiable in any $\bx\in \ker(\bA)\cap \mathcal D=C$ and that the first derivatives are continuous in any $\bx\in C$. 
		
		To prove the differentiability of $\bg$ in any $\bx\in C$ we make use of Theorem \ref{Thm:gDiff}, and hence we have to prove the following.
		\begin{enumerate}[label=\arabic*.]
			\item The map $\bg$ is continuous in any $\bx\in C$.
			\item The map $\bF$ is differentiable in $(\bx,\bg(\bx))$ for all $\bx\in C$.
			\item The Jacobian $\bD_\by\bF(\bx,\bg(\bx))$ with respect to $\by$ is invertible for all $\bx\in C$.
		\end{enumerate}
		If we have shown these properties, then Theorem \ref{Thm:gDiff} together with the considerations above implies that \eqref{eq:Dg(x)} even holds for all $\bx\in \mathcal D$. 		
		
		We first prove that $\bg$ is continuous on $C$. Since gBBKS schemes conserve all linear invariants, we find from \eqref{PDS_test} that
		\begin{equation}\label{eq:gnorm}
			\min\{1,c\}\Vert \bg(\bx)\Vert_1\leq  g_1(\bx)+cg_2(\bx)=x_1+cx_2\leq \max\{1,c\} \Vert \bx\Vert_1.
		\end{equation}
		Now, $\Vert\bx\Vert_1$ is bounded on a sufficiently small neighborhood $\mathcal D$ of $\by^*$ as we can make sure that the closure of $\mathcal D$ is contained in the domain of $\bg$. And since norms on $\R^2$ are equivalent, we \suggest{even} find \suggest{from \eqref{eq:gnorm}} that $\Vert\bg\Vert$ is bounded on $C$. As a result\suggest{,} $H(\cdot,\bg(\cdot))$ is bounded on $\mathcal D$ \suggest{since} the reciprocal of $\bm \sigma\in \mathcal C^2$ as well as $r\in \mathcal C^2$ are bounded on a sufficiently small $\mathcal D$. It thus follows that $\bA\bx H(\bx,\bg(\bx))$ tends to $\bzero$ as $\bx\to \by^*$. From \eqref{eq:F(x).gBBKS1} with $\by=\bg(\bx)$ we therefore obtain \[\lim_{\bx\to\by^*}\bg(\bx)=\by^*=\bg(\by^*), \]
		which means that $\bg\colon D\to D$ is continuous in all $\bx\in C$.
		
		Next\suggest{,} we show that $\bF$ is differentiable in $(\bx,\bg(\bx))$ for all $\bx\in C$.
		For this consider a $\bx_0\in C$ and set $\by_0=\bg(\bx_0)$.
		Note that $\Psi=H(\cdot,\by_0)$ is continuous in $\bx_0\in C$ with $\Psi(\bx_0)=1$ since $\bg(\bx_0)=\bm\sigma(\bx_0)=\bx_0$.
		
		In this case\suggest{,} part \ref{it:alemdiff} of Lemma \ref{Lem:diff} \suggest{from the appendix} with $\mathbf \Phi(\bx)=\bA\bx$ yields
		\begin{equation}\label{eq:DxF=-I-dtA}
			\bD_\bx\bF(\bx_0,\bg(\bx_0))=-\bI-\Delta t\bA.
		\end{equation}
		Furthermore, as $\bA\bx\widetilde H_i(\bx,\by)=\bzero$ for all $\bx\in C$ and $\by\in\R^2_{>0}$, it follows immediately that 
		\begin{equation}\label{eq:DyF=I}
			\bD_\by\bF(\bx_0,\bg(\bx_0))=\bI,
		\end{equation}
		which shows that $\bF$ is partially differentiable in $(\bx_0,\bg(\bx_0))$. To prove that $\bF$ is differentiable in $(\bx_0,\bg(\bx_0))$, we show that the partial derivatives are continuous in $(\bx_0,\bg(\bx_0))$.
		Therefore, we consider the case $\bx\notin C$ and differentiate $\bF$ from \eqref{eq:F(x).gBBKS1} with respect to $\bx$ and $\by$. We have
		
		\begin{equation}\label{eq:DxFtilde(x)}
			\bD_\bx\bF(\bx,\bg(\bx))=-\bI-\Delta t\left(\bA\widetilde H_i(\bx,\bg(\bx))+\bA\bx \nabla_\bx\widetilde H_i(\bx,\bg(\bx))\right),
		\end{equation}
		where the gradient denotes a row vector and
		\begin{equation}\label{eq:i1}
			i=\begin{cases}
				1, &x_1>\frac{b}{a}x_2, \\
				2, &x_1<\frac{b}{a}x_2.
			\end{cases}
		\end{equation}
		Now, since $\bg,\bm \sigma>\bzero$ we \suggest{can} write $\widetilde H_i(\bx,\bg(\bx))=e^{r(\bx)\ln\left(\frac{g_i(\bx)}{\sigma_i(\bx)}\right)}$, from which it follows that 
		\begin{equation}\label{eq:nabla(g/sigma)}
			\nabla_\bx\widetilde H_i(\bx,\bg(\bx))=\widetilde H_i(\bx,\bg(\bx))\left(\nabla_\bx r(\bx)\ln\left(\frac{g_i(\bx)}{\sigma_i(\bx)}\right)-r(\bx)\frac{\nabla_\bx \sigma_i(\bx)}{\sigma_i(\bx)}\right)
		\end{equation}
		since
		\begin{equation*}
			\nabla_\bx \ln\left(\frac{y_i}{\sigma_i(\bx)}\right)=
			\nabla_\bx \ln(y_i)-\nabla_\bx\ln(\sigma_i(\bx))=-\frac{\nabla_\bx \sigma_i(\bx)}{´\sigma_i(\bx)}.
		\end{equation*}
		Plugging \eqref{eq:nabla(g/sigma)} into \eqref{eq:DxFtilde(x)}\suggest{,} we find
		\begin{equation}\label{eq:DxFtildegbbks1}
			\bD_\bx\bF(\bx,\bg(\bx))=-\bI-\Delta t\widetilde H_i(\bx,\bg(\bx))\left(\bA+\bA\bx\left(\nabla_\bx r(\bx)\ln\left(\frac{g_i(\bx)}{\sigma_i(\bx)}\right)-r(\bx)\frac{\nabla_\bx \sigma_i(\bx)}{\sigma_i(\bx)}\right) \right).
		\end{equation}
		Furthermore, $\bA\bx_0=\bzero$ for $\bx_0\in C$ together with $\sigma_1,\sigma_2,r\in \mathcal C^2$ as well as equation \eqref{eq:DxFtildegbbks1} yield
		\begin{equation}\label{eq:limDxF=-I-dtA}
			\lim_{\bx\to\bx_0}\bD_\bx\bF(\bx,\bg(\bx))=-\bI-\Delta t\lim_{\bx\to\bx_0}\widetilde H_i(\bx,\bg(\bx))\bA=-\bI-\Delta t\bA.
		\end{equation}
		Moreover, due to \eqref{eq:DyFtilde=I-ve_i} and since $\bv^{(i)}$ is continuous with $\bv^{(i)}(\bx_0)=\bzero$ for $\bx_0\in C$, we find
		\begin{equation}\label{eq:limDyF=I}
			\lim_{\bx\to \bx_0}	\bD_\by\bF(\bx,\bg(\bx))=\bI-\bv^{(i)}(\bx_0)\be_i^\mathsf{T}=\bI.
		\end{equation}
		As a result of \eqref{eq:DxF=-I-dtA}, \eqref{eq:limDxF=-I-dtA} and \eqref{eq:DyF=I}, \eqref{eq:limDyF=I}\suggest{,} we thus know that all partial first derivatives of $\bF$ are continuous in $(\bx_0,\bg(\bx_0))$ for all $\bx_0\in C$, which implies that $\bF$ is differentiable in $(\bx_0,\bg(\bx_0))$ for all $\bx_0\in C$.
		
		Finally, due to \eqref{eq:DyF=I} we know that $\bD_\by\bF(\bx_0,\bg(\bx_0))$ is invertible for all $\bx_0\in C$.
		
		Altogether\suggest{,} all requirements of Theorem \ref{Thm:gDiff} are fulfilled, which implies that $\bg$ is differentiable on $C$ and that 
		\[\bD\bg(\bx_0)=-(\bD_\by\bF(\bx_0,\bg(\bx_0)))^{-1}\bD_\bx\bF(\bx_0,\bg(\bx_0))\]
		holds for all $\bx_0\in C$. Moreover, all entries of the inverse of $\bD_\by\bF(\bx,\bg(\bx))$ are continuous functions of $\bx$, which proves that $\bg\in \mathcal C^1(\mathcal D)$. Finally, \eqref{eq:DxF=-I-dtA} and \eqref{eq:DyF=I} yield \[\bD\bg(\by^*)=\bI+\Delta t\bA.\]

		\item
		In this part, we use the equations \eqref{eq:DxFtildegbbks1} and 	\begin{equation}\label{eq:DyFtildeInversegBBKS1}
			(\bD_\by\bF(\bx,\bg(\bx)))^{-1}=\begin{cases}
				\frac{1}{1-v_1^{(1)}(\bx)}\begin{pmatrix*}[r]1 & 0\\ v_2^{(1)}(\bx) & 1-v_1^{(1)}(\bx)\end{pmatrix*}, & x_1>\tfrac{b}{a}x_2,\\
				\frac{1}{1-v_2^{(2)}(\bx)}\begin{pmatrix*}[r]1- v_2^{(2)}(\bx) &v_1^{(2)}(\bx)\\  0& 1\end{pmatrix*}, & x_1<\tfrac{b}{a}x_2\end{cases}
		\end{equation} to show that the first derivatives of $\bg$ are Lipschitz continuous on a sufficiently small neighborhood $\mathcal D$ of $\by^*$. For this, we make use of the fact that the set of bounded Lipschitz continuous functions is closed under summation, multiplication and composition. Hence, all we need to prove is that each entry in the matrices \eqref{eq:DxFtildegbbks1} and \eqref{eq:DyFtildeInversegBBKS1} is bounded and Lipschitz continuous on $\mathcal D$\suggest{,} and to use the fact that the natural logarithm and each exponential function are locally Lipschitz continuous. 
		
		To bound the corresponding functions, we choose $\mathcal D$ in such a way that $g_i,\sigma_i$ and $1-v_i^{(i)}$ have an upper bound $C_i>0$ and lower bound $c_i>0$. This is possible by choosing $\bar{\mathcal D}\tm D$ since these functions are continuous at $\by^*$ and satisfy $\bg(\by^*)=\bm \sigma(\by^*)=\by^*>\bzero$ as well as $1-v_i^{(i)}(\by^*)=1$. As a result\suggest{,} even the first two derivatives of $\bm \sigma$ and $r$ are bounded on $\mathcal{D}$. This way\suggest{,} we can compute the Lipschitz constants of $\bm \sigma$, its first derivatives and its reciprocal by using the mean value theorem.
		Analogously, $\bg$ as well as $\frac{1}{g_i}$ are bounded Lipschitz continuous functions \suggest{for $i=1,2$} as their first derivatives are bounded on $\mathcal{D}$. By this reasoning, it is straightforward to verify that each matrix entry in \eqref{eq:DxFtildegbbks1} and \eqref{eq:DyFtildeInversegBBKS1} is a bounded Lipschitz continuous function.
		
	\end{enumerate}
\end{proof}
As a result of Theorem \ref{Thm_MPRK_stabil} and Theorem \ref{Thm:_Asym_und_Instabil} we obtain the following statements due to $\bD\bg(\by^*)=\bI+\Delta t\bA$.
\begin{cor}\label{cor:gBBKS1}
	Let $\by^*>\bzero$ be an arbitrary steady state of \eqref{PDS_test}. Under the assumptions of Theorem \ref{Thm:gBBKS1}, the gBBKS1 schemes have the same stability function as the underlying Runge--Kutta method, i.\,e.\ $R(z)=1+z$ and the following holds.
	\begin{enumerate}
		\item If $\abs{R(-(ac+b)\Delta t)}<1$, then $\by^*$ is a stable fixed point of each gBBKS1 scheme and there exists a $\delta>0$, such that $\vec{1\\ c}^\mathsf{T} \by^{0}=\vec{1\\ c}^\mathsf{T}\by^*$ and $\norm{\by^0-\by^*}<\delta$ imply $\by^n\to \by^*$ as $n\to \infty$.
		\item If $\abs{R(-(ac+b)\Delta t)}>1$, then $\by^*$ is an unstable fixed point of each gBBKS1 scheme.
	\end{enumerate}
\end{cor}
\subsection{Stability of second order gBBKS schemes}
In this subsection we investigate  the gBBKS2($\alpha$) schemes \eqref{eq:gBBKS2Intro} applied to \eqref{PDS_test}, which can be written in the form
\begin{subequations}\label{eq:gBBKS22b}
	\begin{align}\label{eq:stage1_gBBKS22b}
		&\begin{aligned}
			\mathllap{y_i^{(2)}} &= y_i^n+\alpha \Delta t (\bA\by^n)_i \Bigg(\prod_{j\in J^{n}} \frac{y^{(2)}_j}{\pi^{n}_j}\Bigg)^{\ns q^{n}},
		\end{aligned}\\ 
		&\begin{multlined}[b][.7\columnwidth]
			\mathllap{y_i^{n+1}} = y_i^n+\Delta t \left(\Big(1-\frac{1}{2\alpha}\Big) (\bA\by^n)_i+ \frac{1}{2\alpha}(\bA\by^{(2)})_i\right)\Bigg(\prod_{m\in M^{n}} \frac{y^{n+1}_m}{\sigma^{n}_m}\Bigg)^{\ns r^{n}},\label{eq:finalstage_gBBKS22c}
		\end{multlined}  
	\end{align}
\end{subequations}
for $i=1,2$, $\alpha\geq \tfrac12$ and
\begin{subequations}\label{eq:sets22}
	\begin{align}\label{eq:Jset22}
		J^{n}&=\left\{ j\in \{1,2\}\mid (\bA\by^n)_j<0\right\},\\\label{eq:Mset22}
		M^{n}&=\left\{ m\in \{1,2\}\;\Big |\;\Big(1-\frac{1}{2\alpha}\Big) (\bA\by^n)_m+ \frac{1}{2\alpha}(\bA\by^{(2)})_m<0\right\}.
	\end{align}
\end{subequations}
Similarly to the gBBKS1 case\suggest{,} we introduce functions $r,q,\bm \pi$ and $\bm \sigma$ to describe the dependence of the parameters on $\by^n$. Note that $\bm \sigma$ can depend on $\by^n$ as well as $\by^{(2)}$, see \cite{MR4109346,BBKS2007,BRBM2008}, and thus will be described by a map $\bm\sigma\colon\R^2_{>0}\times \R^2_{>0}\to \R^2_{>0}$. 
\begin{thm}\label{Thm:gBBKS2}
	Let $\pi_1,\pi_2,r,q\in \mathcal{C}^2(\R^2_{>0},\R_{>0})$, $\bm \sigma\in \mathcal{C}^2(\R^2_{>0}\times \R^2_{>0},\R^2_{>0})$ and $\by^*>\bzero$ be a steady state solution of \eqref{PDS_test}. Also, let $\mathcal D$ be a sufficiently small neighborhood of $\by^*$ and suppose that $\bm \sigma(\bv,\bv)=\bm \pi(\bv)=\bv$ for all $\bv\in C= \ker(\bA)\cap \mathcal D$. Then \suggest{the} map $\bg$ generating the iterates of the gBBKS2($\alpha$) family satisfies $\bg(\bv)=\bv$ for all steady states $\bv\in  C$ and the following statements are true.
	\begin{enumerate}
		\item\label{it:aThm} The map $\bg$ satisfies $\bg\in \mathcal{C}^1(\mathcal D)$ and $\bD\bg(\by^*)=\bI+\Delta t\bA+\frac{(\Delta t)^2}{2}\bA^2$.
		\item  The first derivatives of $\bg$ are bounded and Lipschitz continuous on $\mathcal D$.
	\end{enumerate}
\end{thm}
\begin{proof}
	Our main strategy is to follow the ideas used in the proof of  Theorem \ref{Thm:gBBKS1}. For this, we first compute the sets $J^n$ and $M^n$ in the case of the linear test problem \eqref{PDS_test}. Using \eqref{eq:stage1_gBBKS22b}, we obtain 
	\[ \Big(1-\frac{1}{2\alpha}\Big) (\bA\by^n)_m+ \frac{1}{2\alpha}(\bA\by^{(2)})_m=(\bA\by^n)_m\left(1+\alpha \Delta t\Bigg(\prod_{j\in J^{n}} \frac{y^{(2)}_j}{\pi^{n}_j}\Bigg)^{\ns q^{n}}\right),\]
	so that\[M^n=J^n=\begin{cases}
		\{1\}, & y_1^n>\frac{b}{a}y^n_2,\\
		\emptyset, & y_1^n=\frac{b}{a}y^n_2,\\
		\{2\}, & y_1^n<\frac{b}{a}y^n_2
		
	\end{cases}\] follows as in the case of gBBKS1.  Next, we define 
	\begin{equation}
		\begin{aligned}
			\by^{(2)}(\bx)&=\bx-\Delta t\alpha\bA\bx\begin{cases}
				\left(\frac{y_1^{(2)}(\bx)}{\pi_1(\bx)}\right)^{q(\bx)},& x_1>\frac{b}{a}x_2,\\
				1, & x_1=\frac{b}{a}x_2,\\
				\left(\frac{y_2^{(2)}(\bx)}{\pi_2(\bx)}\right)^{q(\bx)},& x_1<\frac{b}{a}x_2
			\end{cases} \label{eq:1.gBBKS2}
		\end{aligned}
	\end{equation}
	and
	\begin{equation}\label{eq:defFgbbks2}
		\bF(\bx,\by)=\by-\bx-\Delta t\left(\Big(1-\frac{1}{2\alpha}\Big)\bA\bx+\frac{1}{2\alpha}\bA\by^{(2)}(\bx)\right)H(\bx,\by),
	\end{equation}
	where
	\begin{equation*}
		H(\bx,\by)=\begin{cases}
			\widetilde H_1(\bx,\by),& x_1>\frac{b}{a}x_2,\\
			1, & x_1=\frac{b}{a}x_2,\\
			\widetilde H_2(\bx,\by),& x_1<\frac{b}{a}x_2
		\end{cases}
	\end{equation*}
	as well as
	\begin{equation}\label{eq:HtildegBBKS2}
		\widetilde H_i(\bx,\by)=\left(\frac{y_i}{\sigma_i(\bx,\by^{(2)}(\bx))}\right)^{r(\bx)},\quad i=1,2,
	\end{equation}
	and point out that the function $\bg$ generating the gBBKS2($\alpha$) iterates is the unique solution to 
	\begin{equation}
		\begin{aligned}\label{eq:0=FtildegBBKS2}
			\bzero&=\bF(\bx,\bg(\bx)).
		\end{aligned}
	\end{equation}
	Note that equation \eqref{eq:1.gBBKS2} represents the gBBKS1 schemes applied to \eqref{PDS_test} with a time step size of $\Delta t\alpha$. Hence, Theorem \ref{Thm:gBBKS1} implies that the function $\by^{(2)}$ is a $\mathcal{C}^1$-map on $\mathcal D$ with Lipschitz continuous first derivatives and \[\bD\by^{(2)}(\by^*)=\bI+\Delta t\alpha \bA.\]
	Furthermore, $\bv\in\ker(\bA)$ implies $\by^{(2)}(\bv)=\bv$, and thus, inserting $\bx=\bv$ into equation \eqref{eq:0=FtildegBBKS2} yields $\bg(\bv)=\bv$.
	\begin{enumerate}
		\item Along the same lines as in the proof of Theorem \ref{Thm:gBBKS1} we see that the map $\bF$ is not differentiable on $\mathcal D\times \mathcal D$ since $\bm \sigma(\bx_0,\by^{(2)}(\bx_0))=\bx_0$ holds for all $\bx_0\in  C$. However, $\bF$ is differentiable in $(\bx,\by)\in E=\mathcal D\setminus\ker(\bA)\times \mathcal D$ since 
		\begin{equation}\label{eq:FgBBKS2}
			\bF(\bx,\by)=\by-\bx-\Delta t\left(\Big(1-\frac{1}{2\alpha}\Big)\bA\bx+\frac{1}{2\alpha}\bA\by^{(2)}(\bx)\right)\left(\frac{y_i}{\sigma_i(\bx,\by^{(2)}(\bx))}\right)^{r(\bx)},\quad i=\begin{cases}
				1&, x_1>\frac{b}{a}x_2,\\
				2&, x_1<\frac{b}{a}x_2
			\end{cases}
		\end{equation}
		and $r\in \mathcal{C}^2(\R^2_{>0},\R_{>0})$, $\bm \sigma\in \mathcal{C}^2(\R^2_{>0}\times \R^2_{>0},\R^2_{>0})$ as well as $\by^{(2)}\in\mathcal C^1(\mathcal D)$. Following the proof of Theorem \ref{Thm:gBBKS1}, we show that $\bD_\by\bF(\bx,\bg(\bx))$ is nonsingular in order to show that $\bg\in \mathcal C^1$ on $\mathcal D\setminus\ker(\bA)$.
		First note that for $\bx\notin C$ we have
		\begin{equation}\label{eq:DyFgbbks2}
			\begin{aligned}
				\bD_\by\bF(\bx,\bg(\bx))&=\bI-\Delta t\left(\Big(1-\frac{1}{2\alpha}\Big)\bA\bx+\frac{1}{2\alpha}\bA\by^{(2)}(\bx)\right)\nabla_\by \widetilde H_i(\bx,\bg(\bx)), \quad i=\begin{cases}
					1 &, x_1>\frac{b}{a}x_2,\\
					2&, x_1<\frac{b}{a}x_2,
				\end{cases}
			\end{aligned}
		\end{equation}
		where \eqref{eq:HtildegBBKS2} yields
		\begin{equation*}
			\nabla_\by \widetilde H_i(\bx,\bg(\bx))=\frac{r(\bx)}{g_i(\bx)}\widetilde H_i(\bx,\bg(\bx))\be_i^\mathsf{T}
		\end{equation*}
		with the $i$-th unit vector $\be_i\in \R^2$  as in the proof of Theorem \ref{Thm:gBBKS1}.
		In order to see that $\bD_\by\bF(\bx,\bg(\bx))$ is invertible, we introduce
		\begin{equation}\label{eq:vtilde}
			\begin{aligned}
				\bv^{(i)}(\bx)&=\Delta t\left(\left(1-\frac{1}{2\alpha}\right)\bA\bx+\frac{1}{2\alpha}\bA\by^{(2)}(\bx)\right)\frac{r(\bx)}{g_i(\bx)}\widetilde H_i(\bx,\bg(\bx))\\
			\end{aligned}
		\end{equation}
		and rewrite \eqref{eq:DyFgbbks2} as
		\begin{equation*}
			\bD_\by\bF(\bx,\bg(\bx))=\bI-\bv^{(i)}(\bx)\be_i^\mathsf{T}.
		\end{equation*}
		Hence, we obtain \[ \det(\bD_\by\bF(\bx,\bg(\bx)))=1-v_i^{(i)}.\] 
		Using \eqref{eq:1.gBBKS2}, we see that 
		\[\bv^{(i)}=\Delta t\bA\bx\left(1+\alpha \Delta t\Bigg( \frac{y_i^{(2)}(\bx)}{\pi_i(\bx)}\Bigg)^{\ns q(\bx)}\right)\frac{r(\bx)}{g_i(\bx)}\widetilde H_i(\bx,\bg(\bx)), \]
		where $(\bA\bx)_i<0$ by definition of the gBBKS2($\alpha$) schemes. As a result we know $v_i^{(i)}<0$, and hence $\det(\bD_\by\bF(\bx,\bg(\bx)))\neq 0$ proving that $\bD_\by\bF(\bx,\bg(\bx))$ is invertible. This together with the corresponding arguments of Theorem \ref{Thm:gBBKS1} implies that $\bg\in \mathcal C^1$ on $\mathcal D\setminus \ker(\bA)$ and 	
		\begin{equation}\label{eq:Dg(x)proofgbbks2}
			\bD\bg(\bx)=-(\bD_\by\bF(\bx,\bg(\bx)))^{-1}\bD_\bx\bF(\bx,\bg(\bx))
		\end{equation}
		for $\bx\in\mathcal D\setminus \ker(\bA)$.
		To apply Theorem \ref{Thm:gDiff}, we proceed as in the proof of  Theorem \ref{Thm:gBBKS1}, i.\,e.\ we have to show that
		\begin{enumerate}[label=\arabic*.]
			\item\label{it:1} the map $\bg$ is continuous in any $\bx\in C$.
			\item\label{it:2} the map $\bF$ is differentiable in $(\bx,\bg(\bx))$ for all $\bx\in C$.
			\item\label{it:3} the Jacobian $\bD_\by\bF(\bx,\bg(\bx))$ with respect to $\by$ is invertible for all $\bx\in C$.
		\end{enumerate}
		
		The continuity of $\bg$ follows along the same lines as in the case of gBBKS1, where we additionally use $\by^{(2)}\in \mathcal C^1(\mathcal D)$ for bounding $H(\cdot,\bg(\cdot))$. 
		
		For proving the differentiability of $\bF$ in $(\bx,\bg(\bx))$ for all $\bx\in C$ we consider an arbitrary \suggest{element} $\bx_0\in C$. Note that $\Psi(\bx)=H(\bx,\bg(\bx_0))$ is continuous in $\bx_0$ with $\Psi(\bx_0)=1$. Furthermore, \[\mathbf \Phi(\bx)=\Big(1-\frac{1}{2\alpha}\Big)\bA\bx+\frac{1}{2\alpha}\bA\by^{(2)}(\bx)\] satisfies $\mathbf \Phi(\bx_0)=\bzero$, which means that part \ref{it:alemdiff} of Lemma \ref{Lem:diff} \suggest{from the appendix} together with $\bD\by^{(2)}(\bx_0)=\bI+\Delta t\alpha\bA$ yield\suggest{s}
		\begin{equation}\label{eq:DxFtildegBBKS2}
			\begin{aligned}
				\bD_\bx\bF(\bx_0,\bg(\bx_0))&=-\bI-\Delta t\left(\Big(1-\frac{1}{2\alpha}\Big)\bA+\frac{1}{2\alpha}\bA\bD\by^{(2)}(\bx_0)\right)\\
				&=-\bI-\Delta t\left(\bA+\frac{\Delta t}{2}\bA^2\right).
			\end{aligned}
		\end{equation}
		Also, since $\mathbf\Phi(\bx_0)H(\bx_0,\by)=\bzero$ for all $\by\in \R^2_{>0}$, we find
		\begin{equation}\label{eq:DyF=IgBBKS2}
			\begin{aligned}
				\bD_\by\bF(\bx_0,\bg(\bx_0))&=\bI,
			\end{aligned}
		\end{equation}
		which shows that $\bF$ is partially differentiable in $(\bx_0,\bg(\bx_0))$. We now prove that the partial derivatives of $\bF$ are also continuous in $(\bx_0,\bg(\bx_0))$, which shows the differentiability of $\bF$ in $(\bx,\bg(\bx))$ for all $\bx\in C$. To that end, we consider $\bx\notin C$ and differentiate $\bF$ from \eqref{eq:FgBBKS2} with respect to $\bx$ and $\by$. First, due to \eqref{eq:DyFgbbks2} and since $\bv^{(i)}$ is continuous with $\bv^{(i)}(\bx_0)=\bzero$ for $\bx_0\in C$, we find
		\begin{equation}\label{eq:limDyF=Igbbks2}
			\lim_{\bx\to \bx_0}	\bD_\by\bF(\bx,\bg(\bx))=\bI-\bv^{(i)}(\bx_0)\be_i^\mathsf{T}=\bI
		\end{equation}
		proving the continuity of the partial derivatives in $(\bx_0,\bg(\bx_0)$ with respect to $\by$. Furthermore, we have 
		\begin{equation*}
			\begin{aligned}
				\bD_\bx\bF(\bx,\bg(\bx))=&-\bI-\Delta t\left(\Big(1-\frac{1}{2\alpha}\Big)\bA+\frac{1}{2\alpha}\bA\bD\by^{(2)}(\bx)\right)\widetilde H_i(\bx,\bg(\bx))\\
				&-
				\Delta t\left(\Big(1-\frac{1}{2\alpha}\Big)\bA\bx+\frac{1}{2\alpha}\bA\by^{(2)}(\bx)\right)\nabla_\bx \widetilde H_i(\bx,\bg(\bx)),
			\end{aligned}
		\end{equation*}
		whose entries converge to those of $\bD_\bx\bF(\bx_0,\bg(\bx_0))$ from \eqref{eq:DxFtildegBBKS2} because of the following. First, we have $\by^{(2)}(\bx_0)=\bx_0$ and $\widetilde{H}_i\in \mathcal C^1(\mathcal D\times \mathcal D)$, which means that the last addend disappears as $\bx\to\bx_0\in C$. Additionally, inserting $\bD\by^{(2)}(\bx_0)=\bI+\alpha\Delta t\bA$ and $\lim_{\bx\to \bx_0}\widetilde{H}_i(\bx,\bg(\bx))=1$ yield \eqref{eq:DxFtildegBBKS2}.

		Finally, it follows from \eqref{eq:DyF=IgBBKS2} that the Jacobian $\bD_\by\bF(\bx,\bg(\bx))$ with respect to $\by$ is invertible for all $\bx\in C$. Hence, Theorem \ref{Thm:gDiff} together with the considerations above proves that $\bg$ is differentiable in all $\bx\in \mathcal D$. 
		
		Moreover, since $\bF$ is continuously differentiable in $(\bx,\bg(\bx))$\suggest{,} we find due to \eqref{eq:Dg(x)proofgbbks2} that even $\bg\in \mathcal C^1(\mathcal D)$ holds \suggest{true}. Furthermore, inserting \eqref{eq:DxFtildegBBKS2} and \eqref{eq:DyF=IgBBKS2} into formula \eqref{eq:Dg(x)proofgbbks2} yields
		\[\bD\bg(\bx)=-(\bD_\by\bF(\bx,\bg(\bx)))^{-1}\bD_\bx\bF(\bx,\bg(\bx))=\bI+\Delta t\bA+\frac{(\Delta t)^2}{2}\bA^2.\]
		\item We know that $\by^{(2)}\in\mathcal C^1$ has Lipschitz continuous first derivatives on $\mathcal D$ and that $\bm \sigma\in \mathcal C^2$. Hence, with
		\begin{equation*}
			\begin{aligned}
				\nabla_\bx\widetilde H_i(\bx,\bg(\bx))=&\widetilde H_i(\bx,\bg(\bx))\Biggl(\nabla_\bx r(\bx)\ln\left(\frac{g_i(\bx)}{\sigma_i(\bx,\by^{(2)}(\bx))}\right)\\
				&-r(\bx)\frac{\nabla_\bx \sigma_i(\bx,\by^{(2)}(\bx))+\nabla_\by \sigma_i(\bx,\by^{(2)}(\bx))\bD\by^{(2)}(\bx)}{\sigma_i(\bx)}\Biggr),
			\end{aligned}
		\end{equation*}
		which is analogous to \eqref{eq:nabla(g/sigma)}, this part can be proven along the same lines as in the proof of part \ref{it:cThmgbbks1} of Theorem \ref{Thm:gBBKS1}.
	\end{enumerate}
\end{proof}
This theorem together with Theorem \ref{Thm_MPRK_stabil} and Theorem \ref{Thm:_Asym_und_Instabil}  allows us to conclude the following statements from $\bD\bg(\by^*)=\bI+\Delta t\bA+\tfrac12(\Delta t\bA)^2$.
\begin{cor}\label{cor:gBBKS2}
	Let $\by^*>\bzero$ be an arbitrary steady state of \eqref{PDS_test}. Under the assumptions of Theorem \ref{Thm:gBBKS2}, the gBBKS2($\alpha$) schemes have the same stability function as the underlying Runge--Kutta method, i.\,e.\ $R(z)=1+z+\frac{z^2}{2}$ and the following holds.
	\begin{enumerate}
		\item If $\abs{R(-(ac+b)\Delta t)}<1$, then $\by^*$ is a stable fixed point of each gBBKS2($\alpha$) scheme and there exists a $\delta>0$, such that $\vec{1\\ c}^\mathsf{T} \by^{0}=\vec{1\\ c}^\mathsf{T}\by^*$ and $\norm{\by^0-\by^*}<\delta$ imply $\by^n\to \by^*$ as $n\to \infty$.
		\item If $\abs{R(-(ac+b)\Delta t)}>1$, then $\by^*$ is an unstable fixed point of each gBBKS2($\alpha$) scheme.
	\end{enumerate}
\end{cor}
To summarize the presented analysis of gBBKS schemes, we conclude that the first and second order gBBKS schemes preserve the stability domain of the underlying Runge--Kutta method while preserving positivity. 

\section{Numerical Experiments}\label{sec:NumTests}
In this section \suggest{we focus on the numerical investigation of GeCo and BBKS schemes.} The latters can be written as gBBKS methods using
\[\sigma_m^n=y_m^n \qta r^n=1\]
in \eqref{eq:gBBKS1Intro}, and
\[\pi_m^n=y_m^n,\quad \sigma_m^n=(y_m^n)^{1-\tfrac1\alpha}(y_m^{(2)})^{\tfrac1\alpha} \qta  q^n=r^n=1\]
in \eqref{eq:gBBKS2Intro}, respectively.
\suggest{In particular, we use $\alpha=1$, so that in the following BBKS1, BBKS2(1) as well as GeCo1 and GeCo2 will be analyzed. }
\subsection{Numerical stability of GeCo \suggest{and BBKS schemes}}
To confirm the stability \suggest{results of Section \ref{sec:GeCo} and Section \ref{sec:BBKS}} numerically, we consider the initial value problem 
\begin{equation}\label{eq:testgeco1}
	\by'=\bA\by, \quad \by^\suggest{0}=(0,3,3,3,4)^\mathsf{T},
\end{equation}
where $\bA$ is the $5\times 5$ Metzler matrix
\begin{equation}\label{eq:Atest}
	\bA=\Vec{
		-4 & 2 & 1 & 2 & 2  \\
		1 & -4 & 1 & 0 & 2 \\
		0 & 0 & -4 & 2 & 0 \\
		2 & 2 & 2 & -4 & 0 \\
		1 & 0 & 0 & 0 & -4
}.
\end{equation}
\suggest{The} spectrum of $\bA$ is given by $\sigma(\bA)=\{0,-5-\sqrt3,-5+\sqrt3,-5-\suggest{\ii},-5+\suggest{\ii}\}\tm \C^-$ including real as well as complex eigenvalues. Furthermore, the kernel of $\bA^\mathsf{T}$ is given by $\ker(\bA^\mathsf{T})=\Span(\bn)$ with
$\bn=(1,1,1,1,1)^\mathsf{T}$. Hence, the total mass $\bn^\mathsf{T}\by(t)=\bn^\mathsf{T} \by^0=13$ is a linear invariant for the system, in correspondence of the initial value $\by(0)=\by^\suggest{0}$. The reference solution of the problem is depicted by the dashed lines in Figures~\ref{Fig:geco1} to \ref{Fig:mBBKS2}.

\suggest{We want to note that even though the stability functions of gBBKS and GeCo2 were obtained by analyzing a $2\times2$ system, we will see that the corresponding stability results are well reflected also for a larger system.}

\subsubsection{Investigation of GeCo schemes}
\suggest{Numerical} solutions obtained by GeCo1 and the \suggest{corresponding} error plots are shown in Figure \ref{Fig:geco1}. In error plot \ref{errgeco11}\suggest{,} the convergence of the numerical solution to the steady state in the long run  can be \suggest{seen}, despite \suggest{the} low accuracy in the short run with \suggest{the comparatively large time step of $\Delta t=1$}. 



\begin{figure}
	\begin{subfigure}[t]{0.495\textwidth}
		\includegraphics[height=0.3\textheight]{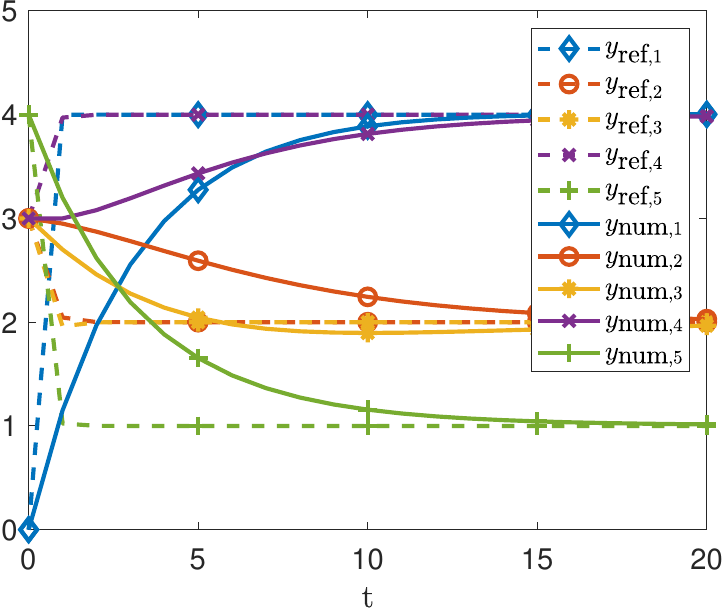}
		\subcaption{$\Delta t=1$ }\label{geco11}
	\end{subfigure}
	\begin{subfigure}[t]{0.495\textwidth}%
		\includegraphics[height=0.295\textheight]{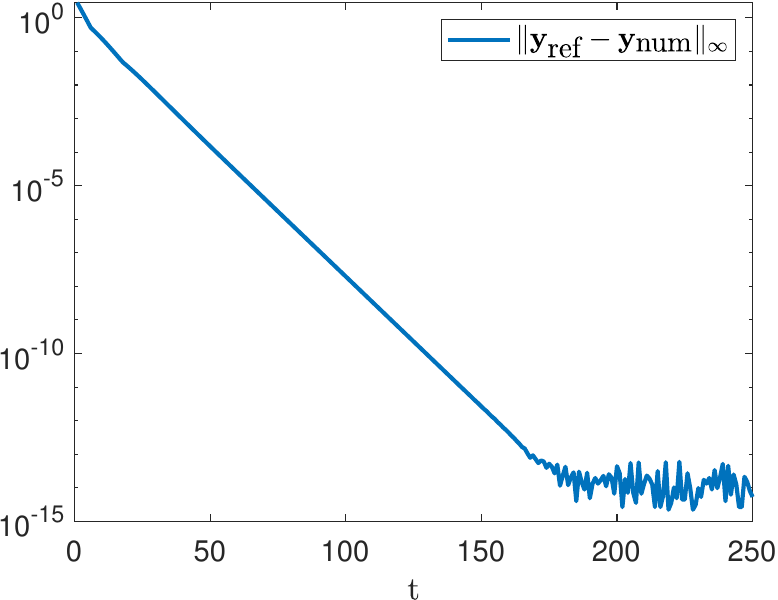}
		\subcaption{$\Delta t=1$}\label{errgeco11}
	\end{subfigure}
	\caption{Numerical approximations of \eqref{eq:testgeco1} and error plot using GeCo1.}
	\label{Fig:geco1}
\end{figure}

\suggest{Based on the analysis for the system \eqref{PDS_test}, we use the function  \[R(z)=1+z+\frac12 z^2\varphi(\Delta t\tr(\bS^-))\]
even in the context of \eqref{eq:testgeco1} to determine the critical time step size $\Delta t_\text{GeCo2}$ of GeCo2.
For the system matrix \eqref{eq:Atest}, we find $\tr(\bS^-)=-\tr(\bA)=20$. A numerical calculation shows that $\abs{R(\Delta t\lambda)}<1$ for all $\lambda\in \sigma(\bA)\setminus\{0\}$ if $\Delta t<\Delta t_{\text{GeCo2}}\approx 0.3572$, where $\Delta t_{\text{GeCo2}}$ was rounded to five significant figures. Moreover, $\abs{R(\Delta t(-5-\sqrt{3}))}>1$ if $\Delta t>\Delta t_{\text{GeCo2}}$. 

In order to numerically confirm the stability results from Theorem~\ref{Thm:geco2} even in the context of the model problem \eqref{eq:testgeco1}, we solve the initial value problem \eqref{eq:testgeco1} using $\Delta t=\Delta t_{\text{GeCo2}}\cdot(1-10^{-3})\approx0.3569.$ The expected stable behavior of GeCo2 and the convergence of the iterates can be observed in Figures \ref{Fig:geco2a} and \ref{Fig:geco2b}. 
In order to demonstrate the expected divergence of the iterates when $\Delta t>\Delta t_\text{GeCo2}$ even for starting vectors that lie within a small neighborhood of the steady state solution, we choose $\Delta t=\Delta t_{\text{GeCo2}}\cdot(1+10^{-3})\approx0.3576$ and the initial value $\widetilde{\by}^0=\by^*+10^{-5}\cdot(-2,1,1,-1,1)^\intercal$.  In Figure \ref{fig:geco2instab},  a small decrease of the error can observed before it increases to an error of approximately $10^{-3}$. Altogether, the numerical experiments reflect the expected behavior independent of $\by^0$, at least for the selected model problem.
 \begin{figure}
	 \begin{subfigure}[t]{0.33\textwidth}
		 \includegraphics[height=0.2\textheight]{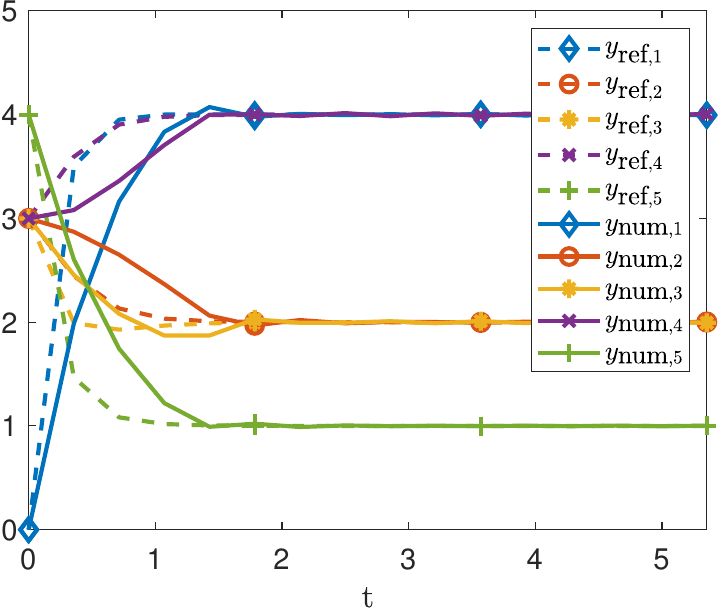}
		 \subcaption{$\Delta t=\Delta t_{\text{GeCo2}}\cdot(1-10^{-3})$} \label{Fig:geco2a}
		 \end{subfigure}
	 \begin{subfigure}[t]{0.33\textwidth}%
		 \includegraphics[height=0.2\textheight]{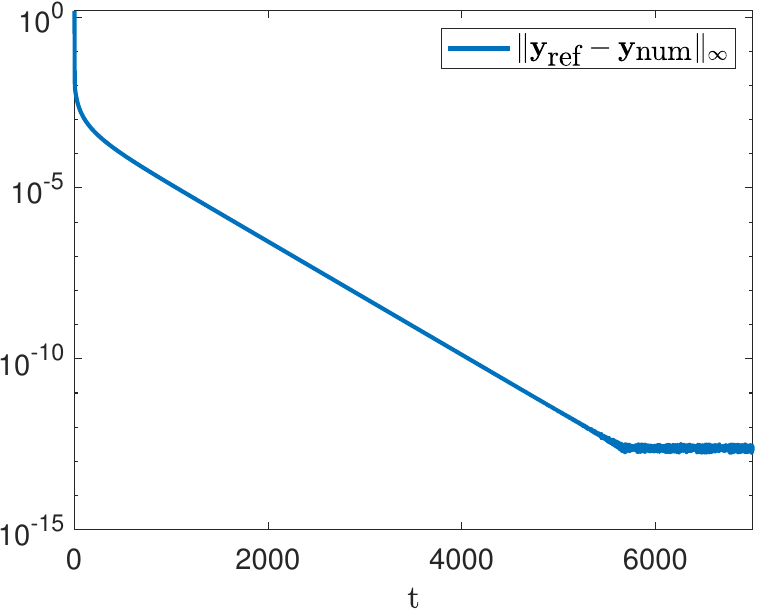}
		 \subcaption{$\Delta t=\Delta t_{\text{GeCo2}}\cdot(1-10^{-3})$}\label{Fig:geco2b}
		 \end{subfigure}
	 \begin{subfigure}[t]{0.33\textwidth}
		 \includegraphics[height=0.197\textheight]{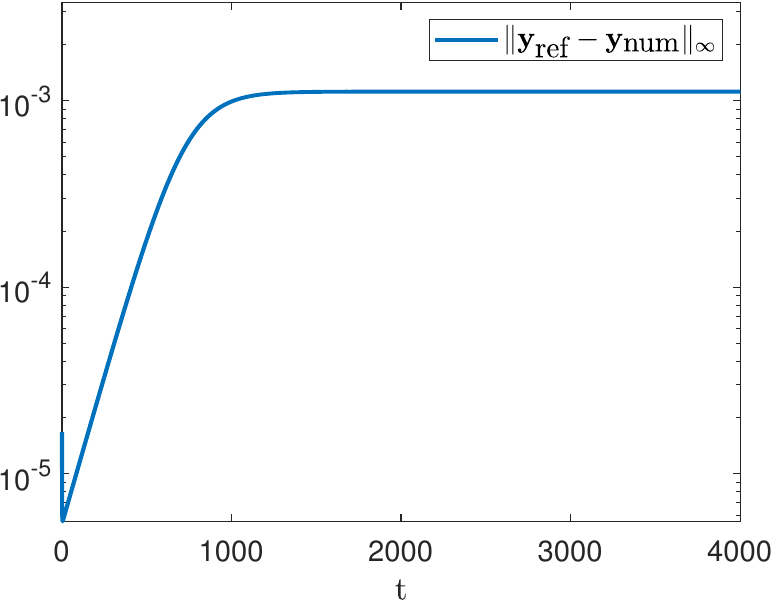}
		 \subcaption{$\Delta t=\Delta t_{\text{GeCo2}}\cdot(1+10^{-3})$ }\label{fig:geco2instab}
		 \end{subfigure}
	 \caption{Numerical approximation of \eqref{eq:testgeco1} and error plots using GeCo2. In (c) the starting vector $\widetilde{\by}^0$ was used.}
	 \label{Fig:geco2}
	 \end{figure}
 
\subsubsection{Investigation of BBKS schemes}
The stability functions of BBKS1 and BBKS2($1$) in the context of \eqref{PDS_test} are given by Theorem \ref{Thm:gBBKS1} and Theorem \ref{Thm:gBBKS2}, respectively. We apply the schemes to the initial value problem \eqref{eq:testgeco1} and test the stability for specific time step sizes. An elementary calculation reveals that the stability functions for both schemes satisfy $\abs{R(\Delta t\lambda)}<1$ for all $\lambda\in \sigma(\bA)\setminus\{0\}$ if $\Delta t < \Delta t_{\text{BBKS}}=\frac{5-\sqrt3}{11},$ and $\abs{R(\Delta t(-5-\sqrt3))}>1$  if $\Delta t>\Delta t_{\text{BBKS}}$. As we did for GeCo2, we investigate the BBKS schemes by varying the time step size around $\Delta t_{\text{BBKS}}$ by multiplying with $1\pm 10^{-3}$, respectively. Furthermore, we also choose $\widetilde{\by}^0=\by^*+10^{-5}\cdot(-2,1,1,-1,1)^\intercal$ in the case $\Delta t>\Delta t_{\text{BBKS}}$ in order to highlight the expected divergence of the iterates.
 
In Figure \ref{Fig:mBBKS1} the numerical solutions of \eqref{eq:testgeco1}  and the error plots using BBKS1 are shown. In \ref{fig:mBBKS1a}, corresponding to the step size $\Delta t=\Delta t_{\text{BBKS}}\cdot(1-10^{-3})\approx 0.2968$, all components of the numerical solution tend to the reference solution in the long run, with an error between $10^{-13}$ and $10^{-12}$.  In the unstable case, see Figure \ref{fig:mBBKS1c}, when $\Delta t =\Delta t_{\text{BBKS}}\cdot(1+10^{-3})\approx 0.2974$, the error increases almost to $10^{-3}$. 
Similar conclusions can be deduced by looking at Figure  \ref{Fig:mBBKS2}, where the numerical solutions and the error plots of BBKS2($1$) are shown, in correspondence of the same step sizes used for BBKS1.

Altogether, the stability properties shown in Figures \ref{Fig:mBBKS1} and \ref{Fig:mBBKS2} are in accordance with the stability results expected from the theory presented in Section \ref{sec:BBKS}.
\begin{figure}[!h]
	\begin{subfigure}[t]{0.33\textwidth}
			\includegraphics[height=0.2\textheight]{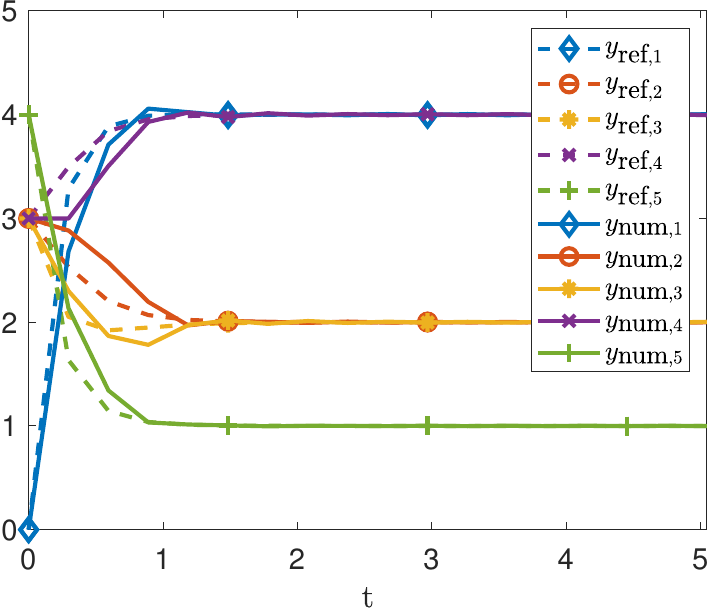}
			\subcaption{$\Delta t=\Delta t_{\text{BBKS}}\cdot(1-10^{-3})$}\label{fig:mBBKS1a}
		\end{subfigure}
	\begin{subfigure}[t]{0.33\textwidth}%
			\includegraphics[height=0.2\textheight]{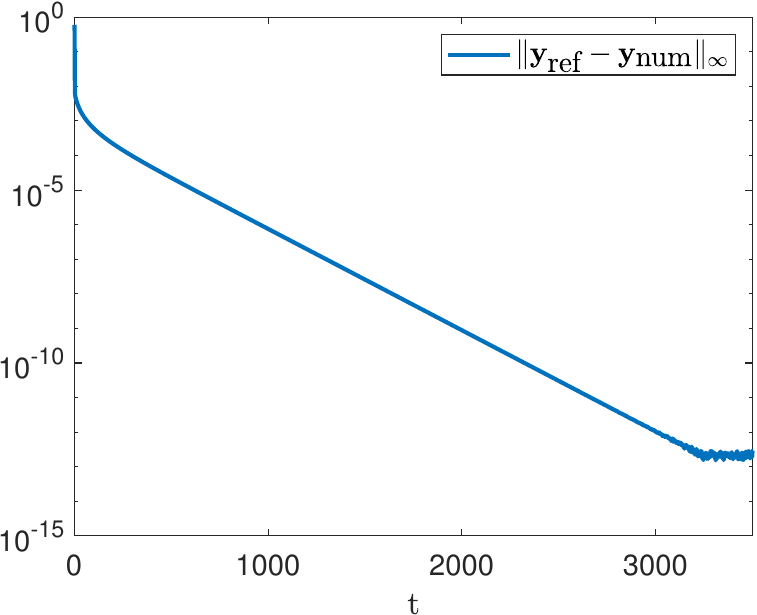}
			\subcaption{$\Delta t=\Delta t_{\text{BBKS}}\cdot(1-10^{-3})$}\label{fig:mBBKS1b}
		\end{subfigure}
	\begin{subfigure}[t]{0.3\textwidth}%
			\includegraphics[height=0.195\textheight]{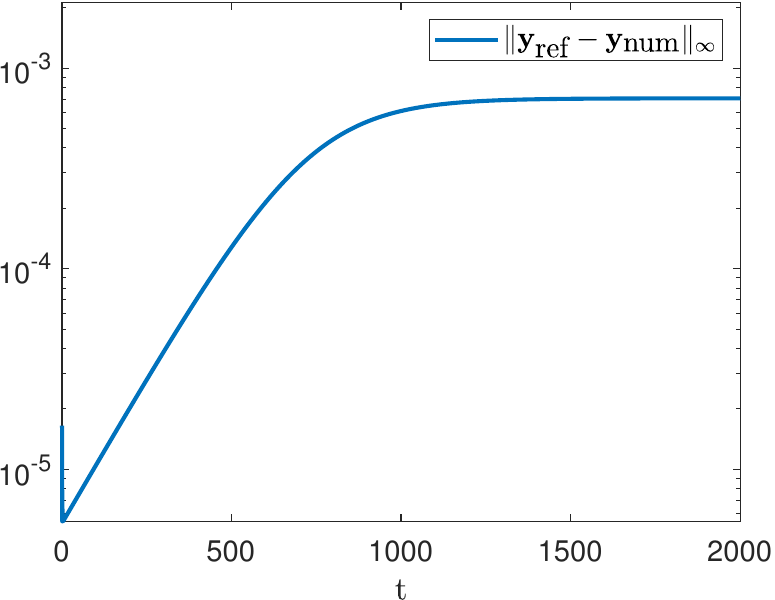}
			\subcaption{$\Delta t=\Delta t_{\text{BBKS}}\cdot(1+10^{-3})$}\label{fig:mBBKS1c}
		\end{subfigure}
	\caption{Numerical approximations of \eqref{eq:testgeco1} and error plots using BBKS1. The starting vector $\widetilde{\by}^0$ was chosen in (c).}
	\label{Fig:mBBKS1}
\end{figure}
\begin{figure}[!h]
	\begin{subfigure}[t]{0.33\textwidth}
		\includegraphics[height=0.2\textheight]{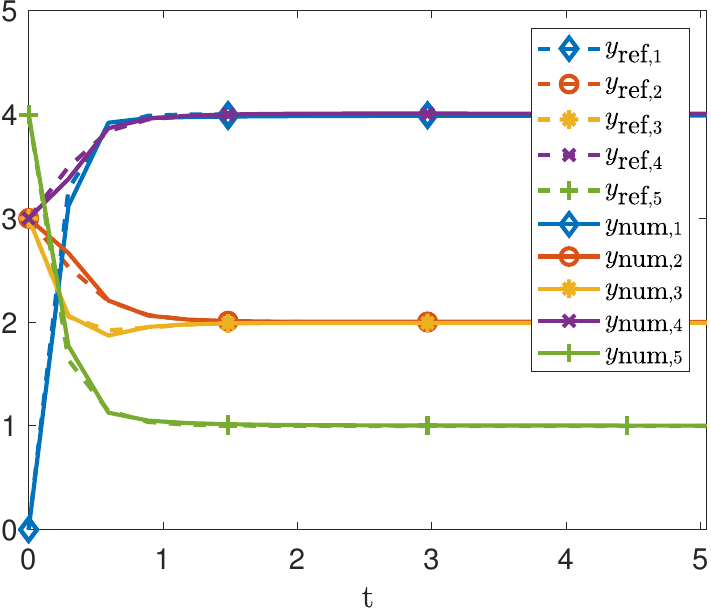}
		\subcaption{$\Delta t=\Delta t_{\text{BBKS}}\cdot(1-10^{-3})$}\label{fig:mBBKS2a}
	\end{subfigure}
	\begin{subfigure}[t]{0.33\textwidth}%
		\includegraphics[height=0.2\textheight]{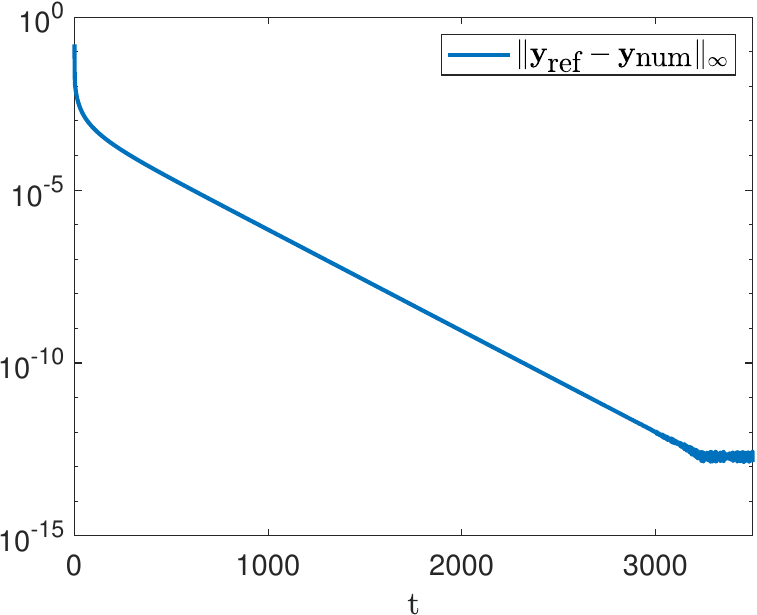}
		\subcaption{$\Delta t=\Delta t_{\text{BBKS}}\cdot(1-10^{-3})$}\label{fig:mBBKS2b}
	\end{subfigure}
	\begin{subfigure}[t]{0.3\textwidth}%
		\includegraphics[height=0.195\textheight]{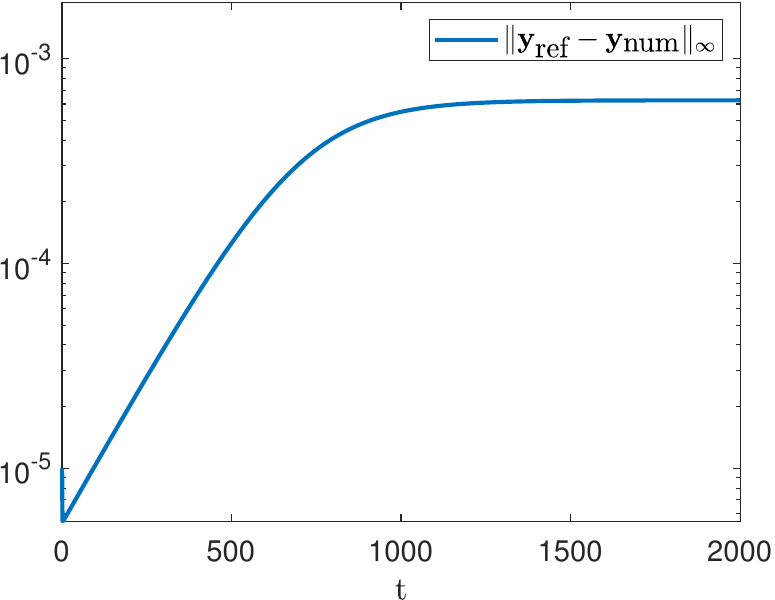}
		\subcaption{$\Delta t=\Delta t_{\text{BBKS}}\cdot(1+10^{-3})$}\label{fig:mBBKS2c}
	\end{subfigure}
	\caption{Numerical approximations of \eqref{eq:testgeco1} and error plots using BBKS2(1). In (c) the starting vector $\widetilde{\by}^0$ was chosen.}
	\label{Fig:mBBKS2}
\end{figure}

}

\suggest{\subsection{Applicability of GeCo1 to stiff problems}
Since the GeCo1 scheme is stable for arbitrary time step sizes, at least locally, this scheme might be able to solve stiff problems.
Unfortunately, this is not true, as we will show next.

To assess the usability for stiff problems, we consider the linear inital value problem $\by'=\bA\by$, $\by(0)=\by^0$ with
\begin{equation}\label{eq:linsys_3x3}
	\bA=\begin{pmatrix}-K & 0 & 0\\\hphantom{-}K & -1 & 0\\0 & \hphantom{-}1& 0\end{pmatrix},\quad \by^0=\begin{pmatrix}0.98\\ 0.01\\ 0.01\end{pmatrix}.
\end{equation}
The solutions of this problem satisfy $y_1(t)+y_2(t)+y_3(t)=1$ for all times $t$ and the system becomes increasingly stiff as the value of $K>0$ is increased. For $K\ne 1$ the solution is
\[y_1(t)=\frac{49e^{-K t}}{50},\quad y_2(t)=\frac{(99K-1)e^{-t}}{100(K-1)}-\frac{49Ke^{-K t}}{50(K-1)},\quad y_3(t)=1 -\frac{(99K-1)e^{-t}}{100(K-1)} +\frac{49e^{-K t}}{50(K-1)}.\]
We see that $y_1$ monotonically tends to $0$, while $y_2$ initially increases before tending to $0$ and $y_3$ monotonically increases  to $1$. 
Defining $\hat\by(t)=\lim_{K\to\infty}\by(t)$ we find
\[\hat y_1(t)=0,\quad \hat y_2(t)=\frac{99}{100}e^{-t},\quad \hat  y_3(t)=1-\frac{99}{100}e^{-t}\]
for $t>0$.
In the limit $K\to\infty$, $y_2$ and $y_3$ should therefore be equal at approximately $t=0.7$. Reference solutions of the problem for different values of $K$ are indicated by the dashed lines in Figure~\ref{fig:linsys_3x3}.
\begin{figure}
	\begin{subfigure}[t]{0.5\textwidth}
		\includegraphics[width=\textwidth]{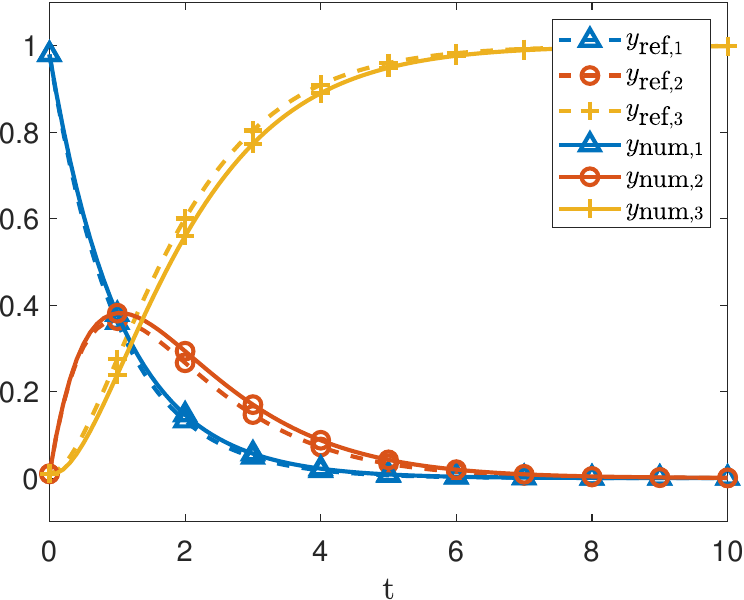}
		\subcaption{$K=1$}\label{subfig:geco1_1_1}
	\end{subfigure}%
	\begin{subfigure}[t]{0.5\textwidth}%
		\includegraphics[width=\textwidth]{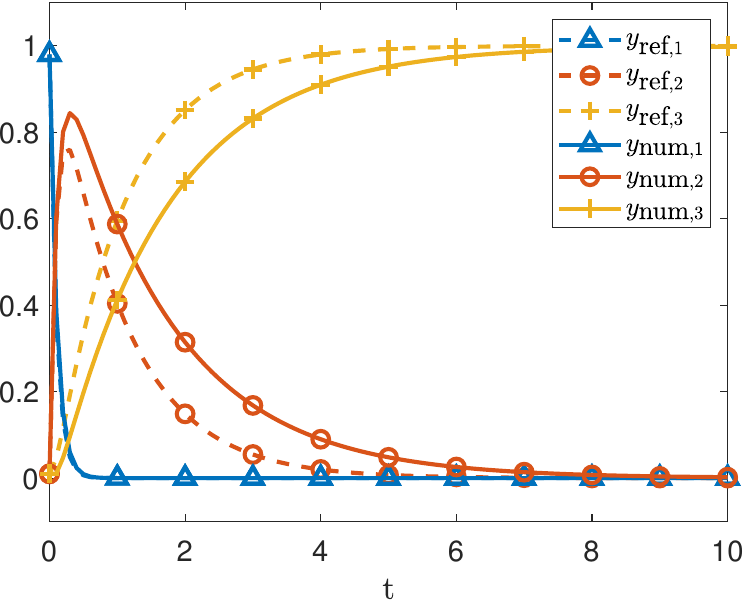}
		\subcaption{$K=10$}\label{subfig:geco1_10_1}
	\end{subfigure}
	
	\begin{subfigure}[t]{0.5\textwidth}
		\includegraphics[width=\textwidth]{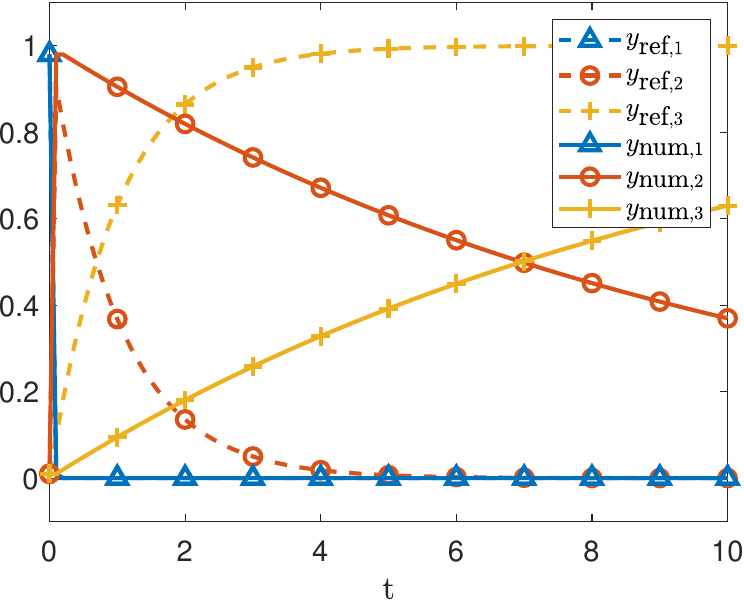}
		\subcaption{$K=100$}\label{subfig:geco1_100_1}
	\end{subfigure}%
	\begin{subfigure}[t]{0.5\textwidth}%
		\includegraphics[width=\textwidth]{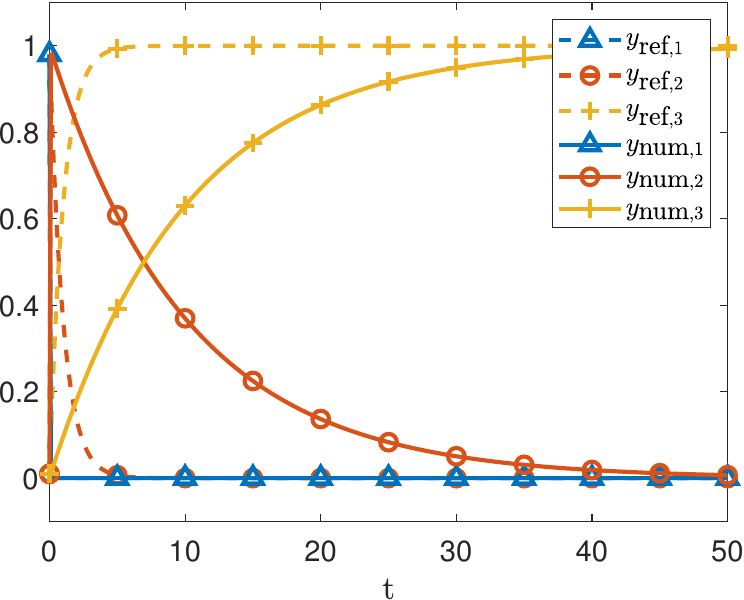}
		\subcaption{$K=100$}\label{subfig:geco1_100_1b}
	\end{subfigure}   
	
	\begin{subfigure}[t]{0.5\textwidth}
		\includegraphics[width=\textwidth]{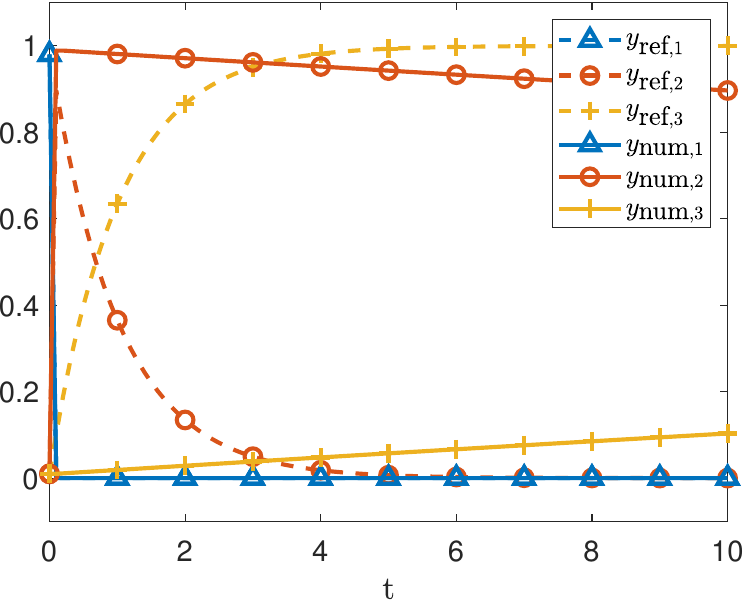}
		\subcaption{$K=1000$}\label{subfig:geco1_1000_1}
	\end{subfigure}%
	\begin{subfigure}[t]{0.5\textwidth}%
		\includegraphics[width=\textwidth]{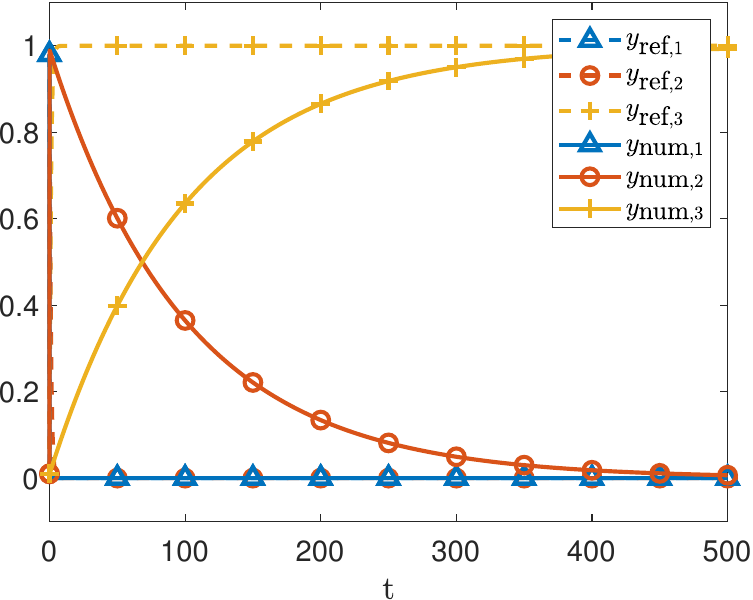}
		\subcaption{$K=1000$}\label{subfig:geco1_1000_1b}
	\end{subfigure}   
	\caption{Numerical solutions of \eqref{eq:linsys_3x3} computed with GeCo1 for different values of $K$. The step size used is $\Delta t=0.1$. The dashed lines indicate the reference solution.}
	\label{fig:linsys_3x3}
\end{figure}

Figure~\ref{fig:linsys_3x3} also shows numerical solutions of \eqref{eq:linsys_3x3} computed with GeCo1 for different values of $K$ and $\Delta t=0.1$. 
We observe that the decrease of $y_1$ is captured quite well independent of the value of $K$, but an increase of $K$ tremendously decreases the accuracy with respect to $y_2$ and $y_3$.
While the general behavior is well captured, a significant phase error is introduced with respect to $y_2$ and $y_3$ that shifts the numerical approximations in such a way that $y_2$ and $y_3$ are equal at about $t=7$ for $K=10$ and about $t=70$ for $K=100$, which is far from $t=0.7$.
Numerical experiments show that to obtain accurate solutions, the step size must be decreased by the same order of magnitude by which $K$ is increased. Hence, the GeCo1 scheme can hardly be regarded as a stiff solver.

Nevertheless, the results shown in Figure~\ref{fig:linsys_3x3} are in accordance with Theorem~\ref{Thm:GeCo1}. We find that the iterates converge to the correct steady state solution, even though a great amount of steps is required. Thus, the statements of the Theorem~\ref{Thm:GeCo1} seem to be valid not only in a small neighborhood of the steady state, at least for this problem.}

\section{Summary and Conclusion}\label{sec:Summary}
In this paper we investigated the stability of gBBKS and GeCo schemes, which preserve all linear invariants of the underlying problem while producing positive approximations for any time step size. These schemes belong to the class of nonstandard methods for which Theorem \ref{Thm_MPRK_stabil} provides a criterion to conclude the Lyapunov stability of a non-hyperbolic fixed point. We found that the gBBKS schemes preserve the stability domain of the underlying Runge--Kutta method while being additionally unconditionally positive. For GeCo2 the investigation revealed that the stability domain is larger than the one of the underlying Heun scheme. Furthermore, GeCo1 possesses stable fixed points for any $\Delta t>0$ and any finite sized linear test problem \eqref{eq:PDS_Sys}, which means that GeCo1 is unconditionally stable while the underlying explicit Euler method is only conditionally stable. 
All theoretical aspects are well reflected in the numerical experiments.

\appendix

\section{Appendix}

\suggest{In this appendix, we present results with rather technical proofs. The first statement provides us conditions under which the product of a scalar continuous function and a partially differentiable vector field is partially differentiable again, and conditions under which a partial derivative of the product does not exist. }
\begin{lem}\label{Lem:diff}
	Let $D\tm \R^N$ be open and $\be_i$ denote the $i$-th unit vector in $\R^N$.  \suggest{Furthermore, let}  
	$\mathbf \Phi\from D\to\R^N$ \suggest{be} partially differentiable in $\mathbf x_0\in D$ with $\mathbf \Phi(\mathbf x_0)=\mathbf 0$ and \suggest{let} $\Psi\from D\to\R$.
	\begin{enumerate}
		\item\label{it:alemdiff}  If $\Psi$ is continuous in $\mathbf x_0$, then the product $\Psi \cdot\mathbf \Phi\from D\to\R^N$ is partially differentiable in $\mathbf x_0$ with
		\[\mathbf D(\Psi\cdot\mathbf \Phi)(\mathbf x_0)=\Psi(\mathbf x_0)\mathbf D\mathbf \Phi(\mathbf x_0).\]
		\item\label{it:blemdiff} If $\Psi(\bx_0+\be_ih)$ possesses several accumulation points as $h\to0$ and $\partial_i\mathbf \Phi(\bx_0)\neq\bzero$, then the $i$-th partial derivative of $\Psi\cdot \mathbf \Phi$ does not exists.
	\end{enumerate}
	
\end{lem}
\begin{proof}
	\begin{enumerate}
		\item Since $\mathbf \Phi(\mathbf x_0)=\mathbf 0$ we have
		\begin{equation}\label{eq:proofLemdiff}
			\begin{aligned}\frac{\Psi(\mathbf x_0 + h\mathbf e_i)\mathbf \Phi(\mathbf x_0 + h\mathbf e_i)-\Psi(\mathbf x_0)\mathbf \Phi(\mathbf x_0)}{h}&=\frac{\Psi(\mathbf x_0 + h\mathbf e_i)\mathbf \Phi(\mathbf x_0 + h\mathbf e_i)-\Psi(\mathbf x_0+h\mathbf e_i)\mathbf \Phi(\mathbf x_0)}{h}\\
				&=\Psi(\mathbf x_0 + h\mathbf e_i)\cdot \frac{\mathbf \Phi(\mathbf x_0 + h\mathbf e_i)-\mathbf \Phi(\mathbf x_0)}{h}.
			\end{aligned}
		\end{equation}
		Passing to the limit $h\to 0$ on both sides shows
		\[\frac{\partial(\Psi\mathbf \Phi)}{\partial x_i}(\mathbf x_0)=\Psi(\mathbf x_0)\frac{\partial \mathbf\Phi}{\partial x_i}(\mathbf x_0),\quad i=1,\dots,N,\]
		and hence
		\[\mathbf D(\Psi\mathbf \Phi)(\mathbf x_0)=\Psi(\mathbf x_0)\mathbf D\mathbf \Phi(\mathbf x_0).\]
		\item If $\Psi(\bx_0+\be_ih)$ possesses several accumulation points as $h\to0$, then this is also true for 
		\[\Psi(\mathbf x_0 + h\mathbf e_i)\cdot \frac{\mathbf \Phi(\mathbf x_0 + h\mathbf e_i)-\mathbf \Phi(\mathbf x_0)}{h}\]
		as $\frac{\partial \mathbf \Phi}{\partial x_i}(\bx_0)\neq\bzero$. As a result of \eqref{eq:proofLemdiff} we thus obtain that $\frac{\partial(\Psi\cdot \mathbf \Phi)}{\partial x_i}(\bx_0)$ does not exist.
	\end{enumerate}
\end{proof}
\suggest{The second and last result of this section is concerned with sufficient conditions for a map $T\colon \R^2_{>0}\to \R$ to be locally Lipschitz continuous even though it is not in $\mathcal C^1$ on its entire domain.}
\begin{lem}\label{Lem:locallyLip}
	Let $\bA\in \R^{2\times 2}$ be given by  \eqref{PDS_test} and define $D_1=\{\bx\in \R^2_{>0}\mid x_1>\frac{b}{a}x_2\}$, $D_2=\{\bx\in \R^2_{>0}\mid x_1<\frac{b}{a}x_2\}$ and $C=\ker(\bA)\cap \R^2_{>0}$. Let $T:\R^2_{>0}\to \R$ be continuous with $T\vert_C=\operatorname{const}$ and $T\vert_{D_i}\in \mathcal C^1$ for $i=1,2$. If $\lim_{\bx\to\bc}\nabla T(\bx)$ exists for any $\bc\in C$, then $T$ is locally Lipschitz continuous.
\end{lem}
\begin{proof}
	
	Note that $C=\partial{D_1}= \partial{D_2}$ and that $T$ is locally Lipschitz on $D_1$ and $D_2$ because $T\vert_{D_i}\in \mathcal C^1$ for $i=1,2$. As a first step, we prove that $T$ is also locally Lipschitz on $\overline{D_i}=D_i\cup C$. For this, we consider closed half balls $H_{\epsilon,i}(\bv)=\overline{B_\epsilon(\bv)}\cap\overline{D_i}$, where $\bv\in C$ and $B_\epsilon(\bv)$ denotes the open ball with center $\bv$ and radius $\epsilon>0$.
	
	As the limit $\lim_{\bx\to\bc}\nabla T(\bx)$ exists for any $\bc\in C$, we can consider the continuous extension of $\nabla T$ to the set $H_{\epsilon,i}(\bv)$, denoted by $\widetilde\bT$. Thus, the mean value theorem and the Cauchy--Schwarz inequality yield
	\begin{equation}\label{eq:HLipschitz}
		\lvert T(\bx_1)-T(\bx_2)\rvert\leq \sup_{\bx\in\overline{B_\epsilon(\bv)}\cap D_i}\Vert \nabla T(\bx)\Vert_2 \Vert \bx_1-\bx_2\Vert_2=\max_{\bx\in H_{\epsilon,i}(\bv)}\Vert \widetilde\bT(\bx)\Vert_2 \Vert \bx_1-\bx_2\Vert_2
	\end{equation}
	for $\bx_1,\bx_2\in \overline{B_\epsilon(\bv)}\cap D_i$, which means that $T$ is Lipschitz continuous on $\overline{B_\epsilon(\bv)}\cap D_i$ for $i\in\{1,2\}$.
	
	Note that $T\vert_C=\operatorname{const}$ implies that $T$ is Lipschitz continuous on $C$. Hence, to prove the Lipschitz continuity on the closed half ball $H_{\epsilon,i}(\bv)$ it remains to consider the case $\bx_1\in C$ and $\bx_2\in \overline{B_\epsilon(\bv)}\cap D_i$ with $i\in \{1,2\}$. For this, we introduce a sequence $(\bx^n)_{n\in \N}\tm \overline{B_\epsilon(\bv)}\cap D_i$ with $\lim_{n\to\infty}\bx^n=\bx_1$. As $T$ is continuous we therefore find $N_0\in \N$ such that for all $n\geq N_0$ we have 
	\begin{equation}\label{eq:H(x^n)-H(x1)<1/n}
		\lvert T(\bx_1)-T(\bx^n)\rvert<\frac{1}{n}.
	\end{equation}
	Altogether, using $L_i=\max_{\bx\in H_{\epsilon,i}(\bv)}\Vert \widetilde\bT(\bx)\Vert_2$ we obtain from \eqref{eq:HLipschitz} and \eqref{eq:H(x^n)-H(x1)<1/n}
	\begin{equation*}
		\lvert T(\bx_1)-T(\bx_2)\rvert\leq \lvert T(\bx_1)-T(\bx^n)\rvert+\lvert T(\bx^n)-T(\bx_2)\rvert< \frac{1}{n}+L_i\Vert \bx^n-\bx_2\Vert_2,
	\end{equation*}
	and passing to the limit, we see that $T$ is even Lipschitz continuous on the closed half ball with a Lipschitz constant $L_i$.
	
	Next, we prove that for any $\bx\in D_1$ and $\by\in D_2$ there exists a $\bz\in C$ such that 
	\begin{equation}
		\Vert \bx-\by\Vert_2=\Vert \bx-\bz\Vert_2+\Vert \bz-\by\Vert_2.\label{eq:triang.equal}
	\end{equation}
	That is to say that  $\bz$ lies on the straight line between $\bx$ and $\by$. Indeed, setting 
	\[ \bz=\bx+c(\by-\bx),\quad c=\frac{x_1-\frac{b}{a}x_2}{x_1-\frac{b}{a}x_2+\frac{b}{a}y_2-y_1},\]
	we find $c\in(0,1)$ as $\bx\in D_1$ and $\by\in D_2$. Additionally, $\bz\in \ker(\bA)$ since
	\[z_1-\frac{b}{a}z_2=x_1+c(y_1-x_1)-\frac{b}{a}(x_2+c(y_2-x_2))=x_1-\frac{b}{a}x_2-c\left(x_1-y_1+\frac{b}{a}y_2-\frac{b}{a}x_2\right)=0, \]
	and $\bz>\bzero$ since it is on the line between $\bx>\bzero$ and $\by>\bzero$.
	
	Let us now prove that $T$ is Lipschitz continuous on $\overline{B_\epsilon(\bv)}$. For this, let $\bx\in D_1$ and $\by\in D_2$, then choose $\bz\in C$ such that \eqref{eq:triang.equal} is satisfied. As a result we obtain
	\begin{equation*}
		\begin{aligned}
			\lvert T(\bx)-T(\by)\rvert&\leq \lvert T(\bx)-T(\bz)\rvert+\lvert T(\bz)-T(\by)\rvert\leq\max\{L_1,L_2\} (\Vert \bx-\bz\Vert_2+\Vert \bz-\by\Vert_2)\\&=\max\{L_1,L_2\}\Vert \bx-\by\Vert_2,
		\end{aligned}
	\end{equation*}
	and since $\bv\in C$ and $\epsilon>0$ are arbitrary, we have proven hat $T$ is locally Lipschitz continuous.
\end{proof}

\section*{Acknowledgements}
The author T.\ Izgin gratefully acknowledges the financial support by the Deutsche Forschungsgemeinschaft (DFG) through the grant ME 1889/10-1. Moreover, Thomas Izgin thanks Stefan Dingel for many fruitful discussions. A. Martiradonna has been partially supported by IndAM-GNCS Project through the grant CUP\_E55F22000270001.

 \bibliographystyle{elsarticle-num} 
 \bibliography{cas-refs}

\end{document}